\documentclass[reqno,12pt]{amsart}
\usepackage{latexsym}
\usepackage{amsmath,amssymb}
\usepackage{amsthm}
\usepackage{amsfonts}
\newtheorem{Definition}{Definition}[part]

\newtheorem{Lemma}{Lemma}[part]
\newtheorem{Proposition}{Proposition}[part]
\newtheorem{Remark}{Remark}[part]
\newtheorem{Theorem}{Theorem}[part]

\numberwithin{Assumption}{section} \numberwithin{Corollary}{section}
\numberwithin{Definition}{section} \numberwithin{equation}{section}
\numberwithin{Example}{section} \numberwithin{Lemma}{section}
\numberwithin{Proposition}{section} \numberwithin{Remark}{section}
\numberwithin{Theorem}{section}

\def\a{\alpha}

\date{}
\title{Global weak solutions to 3D compressible Navier-Stokes-Poisson equations with density-dependent viscosity}
\author [Yulin Ye~~and~~ ]{}

\date{}

\begin{document}
\maketitle
\centerline{\scshape Yulin Ye }

\medskip
{\footnotesize
  \centerline{School of Mathematical Sciences, Capital Normal University}
  \centerline{Beijing  100048, P. R. China}
   \centerline{  \it Email: nkyelin@163.com}
   }
\vspace{3mm}

\pagestyle{myheadings} \thispagestyle{plain}\markboth{} {\small{global weak solutions to Compressible Navier-Stokes-Poisson equations }}

\vspace{3mm}

\textbf{Abstract:}
Global-in-time weak solutions to the Compressible Navier-Stokes-Poisson equations in a three-dimensional torus for large data are considered in this paper. The system takes into account density-dependent viscosity and non-monotone presseur. We prove the existence of global weak solutions to NSP equations with damping term by using the Faedo-Galerkin method and the compactness arguments on the condition that the adiabatic constant satisfies $\gamma>\frac{4}{3}$.


\section{Introduction and Main Results}
 \setcounter{equation}{0}
\setcounter{Assumption}{0} \setcounter{Theorem}{0}
\setcounter{Proposition}{0} \setcounter{Corollary}{0}
\setcounter{Lemma}{0}

In this paper, we consider the following compressible
Navier-Stokes-Poisson equations with density-dependent viscosity
coefficients:

\begin{equation}\label{a1}
\left\{\begin{array}{lll}
\partial_{t}\rho+\rm{div}(\rho u)=0,\\
\partial_{t}(\rho u)+{\rm div}(\rho u\otimes u)+\nabla P(\rho)-{\rm div}(\rho\mathbb{D}u)+r_{1}\rho |u|u=\rho\nabla{\Phi},\\
\lambda\triangle\Phi=4\pi G(\rho -\frac{1}{|\mathbb{T}^{3}|}\int_{\mathbb{T}^{3}}\rho dx),\ \ \lambda=\pm 1,
\end{array}\right.
\end{equation}
with the initial conditions
\begin{equation}\label{a2}
\rho(0,x)=\rho_{0},\rho u(0,x)=m_{0},
\end{equation}
where $t\geq0,x\in\mathbb{T}^{3},\rho=\rho(t,x)$\ and\ $u=u(t,x)$ represent the
fluid density and velocity respectively, $\mathbb{D}u$ is the strain tensor with $\mathbb{D}u=\frac{\nabla u+\nabla u^{\mathrm{T}}}{2}$, and the pressure $P$ is a non-monotone function of the density(see\cite{DFPS2004} for motivations) which satisfies the following conditions:
\begin{equation}\label{a3}
\left\{\begin{array}{lll}
P\in C^{1}(\mathbb{R}_{+}),\ \ P(0)=0,\\
\frac{1}{a}z^{\gamma-1}-b\leq P^{'}(z)\leq a z^{\gamma-1}+b\ \ \ for\ all\ z\geq 0,
\end{array}\right.
\end{equation}
for two constants $a>0$ and $b\geq0$.

The term on the right hand side of the second equation in (\ref{a1}) describes the internal force of gradient vector field produced by potential functions, which can be uniquely solved by the poisson equation$(\ref{a1})_{3}$, yields
$$\lambda\Phi(x)=G\int g(x,y)\rho(y)dy,$$
if and only if
$$\lambda\triangle \Phi=4\pi G(\rho-\frac{1}{|\mathbb{T}^{3}|}\int_{\mathbb{T}^{3}}\rho dx),\ \ in \ \mathbb{T}^{3},\ \ \int_{\mathbb{T}^{3}}\Phi dx=0,$$
where $g=g(x,y)$ denotes the Green's function of the poisson part, $G>0$ is a fixed constant. Moreover, for simplicity, using the conversation of mass $\int_{\mathbb{T}^{3}} \rho dx=\int_{\mathbb{T}^{3}} \rho_{0}dx$, the poisson equation $(\ref{a1})_{3}$ can be normalized as
$$\lambda\triangle\Phi=4\pi G(\rho -1).$$
From a physical point of view, the meaning of the Navier-Stokes-Poisson system is determined by the sign of the parameter $\lambda$. When $\lambda>0$, the potential force $\Phi$ represents the electrostatic potential which produces
the electric field $E=-\nabla\Phi$ and Equations (\ref{a1}) are used to describe the transportation of charged particles in electronic devices, and then $\rho\geq 0,u$ represent the charge density and velocity, respectively. On the other hand, if $\lambda<0$, the potential force $\Phi$ denotes the gravitational force and the Navier-Stokes-Poisson system is used in astrophysics to describe the motion of gaseous stars, and $\rho\geq0,u$ denote the density, velocity of a gaseous star, respectively.

Navier-Stokes-Poisson equations has attracted the attention of many physicists and mathematicians because of its physical importance, rich phenomena, and mathematical challenges. For the case of constant viscosity coefficients, Ducomet and Feireisl in \cite{DF2004} considered the full Navier-Stokes-Poisson equations and proved when $\gamma>\frac{3}{2}$, there exists a global-in-time variational weak solution. In \cite{DFPS2004} Ducomert et al. also proved there exists a global weak solution to the barotropic compressible Navier-Stokes-Poisson equations with no-monotone pressure provided that $\gamma>\frac{3}{2}$. Donatelli \cite{D2003} considered the Cauchy problem for the coupled Navier-Stokes-Poisson equations and gave a positive answer to the existence of local and global weak solutions. Zhang and Tan in \cite{zt2007}, by using the theory of Orlicz spaces, have proved the existence of globally defined finite energy weak solutions for Navier-Stokes-Poisson equations in two dimensions with the pressure satisfying $P(\rho)=a \rho\log^{d}\rho$ for large $\rho$, and $d>1,\ a>0$. Cai and Tan \cite{ct2015} also proved the system has the global weak time-periodic solution for the Navier-Stokes-Poisson equations in a bounded domain with periodic boundary condition as $\gamma>\frac{5}{3}$ when the  external force is time-periodic. Besides, Jiang et al.\cite{jty2009} considered the global behavior of weak solutions of the Navie-Stokes-Poisson equations in a bounded domain with arbitrary forces.

 However, for the case of density-dependent viscosity coefficients, the problem is much more challenge because of the  degeneration near the vacuum and the results are limit. Ducomet et.al in \cite{dnv2010} studied the global stability of the weak solutions to the Navier-Stokes-Poisson equations with the  degenerate viscosities as $\gamma>\frac{4}{3}$, in which the pressure $P(\rho)$ is not necessary a monotone function of the density. Furthermore, in \cite{dnv2011} they also considered Cauchy problem for the Navier-Stokes-Poisson equations of spherically symmetric motions in $\mathbb{R}^{3}$, both constant viscosities and density-dependent viscosities included, and proved the global stability of the weak solutions just provided  that the polytropic index $\gamma$ satisfies $\gamma > 1$. Particularly, if without the drag term $\rho|u|u$ and the poisson term $\rho\nabla\Phi$, the equations (\ref{a1}) will be reduced to the classical barotropic compressible Navier-Stokes equations with degenerate viscosities, especially to be the well known shallow water equation in the dimensions two, the global existence of weak solutions of which is a long standing open problem proposed by Lions \cite{lions1998}. Recently, A.Vasseur and C.Yu have proved in \cite{vy2014} that there exists a global weak solution to the compressible Navier-Stokes equations by constructing some smooth multipliers allowing to derive the Mellet-Vasseur type inequality \cite{mv2007} for the weak solutions, and they also proved the existence of the global weak solutions for the  quantum compressible Navier-Stokes equations, which can be seen as a approximate system to the original equations in \cite{vy2014}, almost at the same time, Li and Xin gave an another approach in \cite{lx2015}. Moreover, it should be noted that the Quasi-neutral limit problem is also considered by many mathematicians for the case of $\lambda>0$, for more details, the readers can refer to \cite{bdd2005,yj2014} and references therein.

 Inspired by \cite{mv2007}, in the present paper, we consider the compressible Navier-Stokes equations with degenerate viscosities, potential force and the damping term, in which the pressure $P(\rho)$ is not necessary a monotone function, and we prove that the problem admits a global weak solution as $\gamma>\frac{4}{3}$ for the case $\lambda=-1$, or $\gamma>1$ for the case $\lambda=1$. Compared to the case of the classical compressible Navier-Stokes equations, the Navier-Stokes-Poisson problem is much more complex and some new difficulties will occur, for example, we can not deduce the energy estimates directly due to the poisson term and non-monotone pressure term, and moreover, when we deduce the energy estimates and the Bresch-Dejardins entropy, the estimates will depend on the index $\varepsilon,\delta$ and $\eta$, so we need to be very careful as we deduce these estimates because we need to tend the $\varepsilon,\eta,\delta$ to zero step by step later in the proof of the main theorem. To our knowledge, this is the first complete proof for the Navier-Stokes-Poisson problem with the degenerate viscosities and here the pressure is not necessary a monotone function of the density, which contains the classical $\gamma$ law case. So our results are much general and can be seen as a supplement of Ducomet et.al\cite{dnv2010} and a extension of \cite{vy2014}, \cite{vy2015}.

 Throughout this paper, we only focus on the case that $\lambda=-1$, and after some small modifications, the method can be directly applied to the case that $\lambda=1$, so we omit the details.
\subsection{ Formulation of the weak solutions and main result}
For the smooth solutions $(\rho,u,\Phi(\rho))$, multiplying the momentum equation $(\ref{a1})_{2}$ and integrating by parts we can deduce the following energy inequality
\begin{equation}\label{a4}
\begin{aligned}
 E(t)+\int_{0}^{T}\int_{\mathbb{T}^{3}} \rho |\mathbb{D}u|^{2}+r_{1}\int_{0}^{T}\int_{\mathbb{T}^{3}}\rho u^{3}\leq E_{0},
 \end{aligned}
\end{equation}
where
 $$E(t)=\int_{\mathbb{T}^{3}}( \frac{1}{2}\rho u^{2}+{\Pi}(\rho)-\frac{1}{8\pi G}|\nabla \Phi|^{2})dx, \ \ \ {\Pi}(\rho)=\rho\int_{1}^{\rho}\frac{P(s)}{s^{2}}ds$$
and
$$E_{0}=\int_{\mathbb{T}^{3}}( \frac{1}{2}\rho_{0} u_{0}^{2}+\Pi(\rho_{0})-\frac{1}{8\pi G}|\nabla \Phi(\rho_{0})|^{2})dx.$$
However, the above energy estimate is not enough to prove the stability of the weak solutions $(\rho,u,\Phi(\rho))$ of $(\ref{a1})$, fortunately, if the viscosity coefficients satisfy a special relation, which means $\lambda(\rho)=\rho\mu(\rho)^{'}-\mu(\rho)$, in this paper it means $\mu(\rho)=\rho$, $\lambda(\rho)=0$, we will obtain the following B-D entropy estimate which was first introduced by Bresch-Desjardins-Lin in \cite{BDL2003}:
\begin{equation}\label{a5}
  \begin{aligned}
  \int_{\mathbb{T}^{3}} \frac{1}{2}&\rho (u+\frac{\nabla \rho}{\rho})^{2} \ dx+\frac{4}{a\gamma^{2}}\int_{0}^{T}\int_{\mathbb{T}^{3}}|\nabla \rho^{\frac{\gamma}{2}}|^{2}dxdt\\
  &+\int_{0}^{T}\int_{\mathbb{T}^{3}}\rho |\nabla u|^{2} dxdt \leq \int_{\mathbb{T}^{3}}\big(\frac{1}{2}\rho_{0}u_{0}^{2}+|\nabla{\sqrt{ \rho_{0}}}|^{2}\big)dx+C,
  \end{aligned}
  \end{equation}
  where C is bounded by the initial energy.
  Thus the initial data should satisfy the follows
  \begin{equation}\label{a6}
  \begin{aligned}
  &\rho_{0}\in L^{1}(\mathbb{T}^{3})\cap L^{\gamma}(\mathbb{T}^{3}),\ \ \rho_{0}\geq 0,\ \ \nabla\sqrt{\rho_{0}}\in L^{2}(\mathbb{T}^{3}),\\
  &m_{0}\in L^{1}(\mathbb{T}^{3}),\ m_{0}=0\ \ if\ \ \rho_{0}=0,\ \ \ \frac{|m_{0}|^{2}}{\rho_{0}}\in L^{1}(\mathbb{T}^{3}).
  \end{aligned}
  \end{equation}
\begin{Definition}\label{def1}
We will say $(\rho, u,\Phi)$ is the finite energy weak solution of the problem (\ref{a1})-(\ref{a2}) if the following is satisfied.
\begin{enumerate}
  \item $\rho$, $u$ belong to the classes
  \begin{equation}\label{a7}
  \left\{\begin{array}{lll}
  \rho\in L^{\infty}((0,T);L^{1}\cap L^{\gamma}(\mathbb{T}^{3})),\ \sqrt{\rho}u\in L^{\infty}((0,T);L^{2}(\mathbb{T}^{3})),\\
  \sqrt{\rho}\in L^{\infty}((0,T);H^{1}(\mathbb{T}^{3})),\ \nabla{\rho^{\frac{\gamma}{2}}}\in L^{2}((0,T);L^{2}(\mathbb{T}^{3})),\\
  \sqrt{\rho}\nabla u\in L^{2}((0,T);L^{2}(\mathbb{T}^{3})),\ \rho^{\frac{1}{3}}u\in L^{3}((0,T);L^{3}(\mathbb{T}^{3}),
  \end{array}\right.
  \end{equation}
  \item
  The equations $(\ref{a1})_{1}$-$(\ref{a1})_{2}$ hold in the sense of $\mathcal{D}^{'}((0,T)\times\mathbb{T}^{3})$,$(\ref{a1})_{3}$ holds a.e. for $(t,x)\in ((0,T)\times \mathbb{T}^{3})$,
  \item
  $(\ref{a2})$ holds in $\mathcal{D}^{'}(\mathbb{T}^{3})$,

  \item
  (\ref{a4}) and (\ref{a5}) hold for almost every $t\in [0,T]$,

\end{enumerate}
\end{Definition}
Before stating the main theorem, we will give some notations:\\

\textbf{Notations:} Throughout this paper, $C$ denotes a generic positive constant which may depend on the initial data or some other constants but independent of the indexes $\varepsilon,\eta,\delta$ and $r_{0}$, $C(\cdot)>0$ means the constant $C$ particularly depends on the parameters in the bracket, and
$$\int f=\int_{\mathbb{T}^{3}} fdx,\ \ \int_{0}^{T}\int f=\int_{0}^{T}\int_{\mathbb{T}^{3}} f dxdt,$$
$$\|f\|_{L^{p}}=\|f\|_{L^{p}(\mathbb{T}^{3})},\ \ \|f\|_{L^{p}(0,T;W^{s,r})}=\|f\|_{L^{p}(0,T;W^{s,r}(\mathbb{T}^{3}))}.$$

Then we will state our main results:
\begin{Theorem}\label{theorem1}
Let $\lambda=-1$, $\gamma >\frac{4}{3}$ and the initial data satisfies(\ref{a6}), then for any time $T$, there exists a weak solution $(\rho, u,\Phi)$ to (\ref{a1})-(\ref{a2}) in the sense of Definition \ref {def1}.

\end{Theorem}
\begin{Remark}
It should be pointed out that the damping term $\rho|u|u$ here is used to give the strong convergence of $\sqrt{\rho}u$ in $L^{2}(0,T;L^{2}(\mathbb{T}^{3})$, so follow the approach in \cite{vy2014}, we can similarly deduce the Mellet-Vasseur inequality for weak solutions, and the damping term can be removed, for this case, we will show the details in our future paper.

\end{Remark}
\begin{Remark}
For the case that $\lambda=1$, with some small modifications, the proof in this paper can be directly extended to this case, and we can also prove that the system (\ref{a1}) admits a global weak solution just provided that $\gamma>1$. So in this paper, we omit the details.
\end{Remark}

The rest paper is organized as follows: in section 2, we state some elementary inequalities and compactness theorems which will be used frequently in the whole proof. To proof our main result, we use the weak compactness analysis method and need to pass to the limits at several approximate levels. In section 3, following the method in \cite{vy2015}, we show the existence of global-in-time weak solutions to the approximate equations by using the Faedo-Galerkin method. In section 4, we deduce the Bresch-Dejardins entropy estimates and pass to the limits as $\varepsilon,\mu\rightarrow0$. In section 5-6, by using the standard compactness arguments, we pass to the limits as $\eta\rightarrow0$ and $\delta\rightarrow0$ step by step.
\section{Preliminaries}
Firstly, we introduce the following Gagliardo-Nirenberg inequality which we will used later when we deduce the energy estimates and B-D entropy.
\begin{Lemma}\label{le1}\cite{n1959}(Gagliardo-Nirenberg interpolation inequality)
For function $u:\Omega\rightarrow \mathbb{R}$ defined on a bounded Lipschitz domain $\Omega\subset \mathbb{R}^{n}$, $\forall 1\leq q,r\leq \infty$ and a natural number $m$, Suppose also that a real number $\alpha$ and a natural number $j $ are such that
$$\frac{1}{p}=\frac{j}{n}+(\frac{1}{r}-\frac{m}{n})\alpha+\frac{1-\alpha}{q}$$
and
$$\frac{j}{m}\leq \alpha\leq 1,$$
then we have
$$\|D^{j}u\|_{L^{p}}\leq C_{1}\|D^{m}u\|_{L^{r}}^{\alpha}\|u\|_{L^{q}}^{1-\alpha}
+C_{2}\|u\|_{L^{s}}$$
where $s>0$ is arbitrary; naturally, the constants $C_{1}$ and $C_{2}$ depend upon the domain $\Omega$ as well as $m,n$ etc.

\end{Lemma}
The following two Lemmas are two standard compactness results and will help us get the strong convergence of the solutions:
\begin{Lemma}\label{le2}\cite{a1963,s1986}(Aubin-Lions Lemma)
Let $X_{0},X$ and $X_{1}$ be three Banach spaces with $X_{0}\subseteq X \subseteq X_{1}$. Suppose that $X_{0}$ is compactly embedded in $X$ and that $X$ is continuously embedded in $X_{1}$. For $1\leq p,q\leq +\infty$, let
$$W=\{u\in L^{p}([0,T];X_{0}) \mid \partial_{t} u\in L^{q}([0,T];X_{1})\}.$$
(i) If $p<+\infty$, then the embedding of $W$ into $L^{p}([0,T];X)$ is compact.\\
(ii) If $p=+\infty$ and $q>1$, then the embedding of $W$ into $C([0,T];X)$ is compact
\end{Lemma}
\begin{Lemma}\label{le3}\cite{nsbook}(Egoroff¡¯s theorem about uniform convergence)Let $f_{n}\rightarrow f$a.e.in $\Omega$, a bounded measurable set in $\mathbb{R}^{n}$, with $f$ finite a.e. Then for any $\varepsilon > 0$ there exists a measurable subset $\Omega_{\varepsilon} \subset \Omega$ such that $|\Omega\backslash \Omega_{\varepsilon}| < \varepsilon$ and $f_{n}\rightarrow f$ uniformly in $\Omega_{\varepsilon}$,
moreover, if
$$f_{n}\rightarrow f\ \ a.e.\ \ in\ \Omega,$$
$$f_{n}\in L^{p}(\Omega)\ \ and\  uniformly\  bounded\ , \ for\  any\   1<p\leq +\infty,$$\\
then, we have
$$f_{n}\rightarrow f \ \ strongly\ \ in\ L^{s},\ \ for\ any\  s\in [1,p).$$
\begin{proof}
Since $f_{n}\rightarrow f$ a.e. in $\Omega$ and $f_{n}$ is uniformly bounded in $L^{p}(\Omega)$, so due to the Egoroffs theorem, we have
$$\forall \varepsilon>0,\ \exists \Omega_{\varepsilon},\ |\Omega-\Omega_{\varepsilon}|<\varepsilon,\ \ \sup_{x\in \Omega_{\varepsilon}}|f_{n}(x)-f(x)|\rightarrow0,\ uniformly\ with\ n,$$
then we can get
\begin{equation*}
\begin{aligned}
\int_{\Omega}|f_{n}-f|^{s}dx&=\int_{\Omega_{\varepsilon}}
|f_{n}-f|^{s}dx+\int_{\Omega-\Omega_{\varepsilon}}|f_{n}-f|^{s}dx\\
&\leq \sup_{x\in\Omega_{\varepsilon}}|f_{n}-f|^{s}|\Omega_{\varepsilon}|
+C\||f_{n}-f\|_{L^{p}}^{s}|\Omega-\Omega_{\varepsilon}|
^{(p-s)\diagup p}\rightarrow 0.
\end{aligned}
\end{equation*}

\end{proof}

\end{Lemma}
\section{Faedo-Galerkin approximation}
In this section, we construct the approximate system to the original problem (\ref{a1}) by using the Faedo-Galerkin method, we proceed similarly in [\cite{feireisl2004}, Chapter. 7] and \cite{jungel2010}.
\subsection{Approximate the mass equation}
Let $T>0$, then we define a finite-dimensional space $X_{n}$=span$\{e_{1},...e_{n}\}$,\ $n\in\mathbb{N}$, where $\{e_{k}\}$ is an orthonormal basis of $L^{2}(\mathbb{T}^{3})$ which is also an orthogonal basis of $H^{1}(\mathbb{T}^{3})$. Let $(\rho_{0},u_{0})\in C^{\infty}(\mathbb{T}^{3})$ be some initial data satisfying $\rho_{0}\geq \nu>0$ for $x\in (\mathbb{T}^{3})$ for some $\nu>0$, and let the velocity $u\in C([0,T];X_{n})$ be given with the following norm
$$u(x,t)=\sum_{i=1}^{n}\lambda_{i}(t)e_{i}(x),\ \ (t,x)\in [0,T]\times\mathbb{T}^{3}$$
Note that $X_{n}$ is a finite-dimensional space, all the norms are equivalence on $X_{n}$, so $u$ is bounded in $C([0,T];C^{k}(\mathbb{T}^{3}))$ for any $k\in \mathbb{N}$ and there exists a constant $C>0$ depending on $k$ such that
\begin{equation}\label{b1}
\|u\|_{C([0,T];C^{k}(\mathbb{T}^{3}))}\leq C\|u\|_{C([0,T];L^{2}(\mathbb{T}^{3}))}.
\end{equation}
Then we approximate the continuity equation as follows:
\begin{equation}\label{b2}
\begin{aligned}
\left\{\begin{array}{lll}
\partial_{t}\rho+\rm{div}(\rho u)=\varepsilon \triangle\rho,\\
\rho_{0}\in C^{\infty}(\mathbb{T}^{3}),\ \ \rho_{0}\geq \nu >0,
\end{array}\right.
\end{aligned}
\end{equation}
Firstly, to show the well-posedness of the parabolic problem (\ref{b2}), we introduce the following Lemma:
\begin{Lemma}\label{lemma1}\cite {fn2009}
Let $\Omega\in \mathbb{R}^{3}$ be a bounded domain of class $C^{2,\nu},\ \nu\in (0,1)$ and let $u\in C([0,T];X_{n})$ be a given vector field. If the initial data $\rho_{0}\geq \nu>0,\ \rho_{0}\in C^{2}(\mathbb{T}^{3})$, then problem (\ref{b2}) possesses a unique classical solution $\rho=\rho_{u}$, more specifically,
\begin{equation}\label{b3}
\begin{aligned}
\rho_{u}\in V\equiv \left\{\begin{array}{lll}
\rho\in C([0,T];C^{2,\nu}(\mathbb{T}^{3})),\\
\partial_{t}\rho\in C([0,T];C^{0,\nu}(\mathbb{T}^{3})).
\end{array}\right\}
\end{aligned}
\end{equation}

\end{Lemma}
Furthermore, because $u\in C([0,T];X_{n})$ is a given vector field, then by using the bootstrap method and above Lemma, it's easy to prove that the system $(\ref{b2})$ exists an unique classical solution $\rho\in C^{1}([0,T];C^{7}(\mathbb{T}^{3}))$.
Moreover, if $0<\underline{\rho}\leq \rho\leq \overline{\rho}$ and ${\rm div}u\in L^{1}([0,T];L^{\infty}(\mathbb{T}^{3}))$, through the maximum principle it provides $$\rho(x,t)\geq 0.$$

Then if we define $L\rho=\partial_{t}\rho+{\rm div}(\rho u)-\varepsilon\triangle\rho$, by direct calculation we can obtain
\begin{equation*}
\begin{aligned}
&L(\overline{\rho}e^{\int_{0}^{T}\|{\rm div}u\|_{L^{\infty}}dt})=\overline{\rho}e^{\int_{0}^{T}
\|{\rm div}u\|_{L^{\infty}}dt}(\|{\rm div}u\|
_{L^{\infty}}+{\rm div}u)\geq0,\\
&L(\underline{\rho}e^{-\int_{0}^{T}\|{\rm div}u\|_{L^{\infty}}dt})\leq 0,\ \ \ \ L\rho=0,
\end{aligned}
\end{equation*}
which means that $\overline{\rho}e^{\int_{0}^{T}\|{\rm div}u\|_{L^{\infty}}dt}$ and $\underline{\rho}e^{-\int_{0}^{T}\|{\rm div}u\|_{L^{\infty}}dt}$ are super and sub solutions to the equation (\ref{b2}) respectively, for $0<\underline{\rho}\leq \rho\leq \overline{\rho}$. So, by using the comparison principle, we can obtain
\begin{equation}\label{b4}
0<\underline{\rho}e^{-\int_{0}^{T}\|{\rm div}u\|_{L^{\infty}}dt}\leq \rho(x,t)\leq \overline{\rho}e^{\int_{0}^{T}\|{\rm div}u\|_{L^{\infty}}dt},\ \ \forall x\in \mathbb{T}^{3},\ \ t\geq0.
\end{equation}

Next we will show that the solution of the equation (\ref{b2}) depends on the velocity $u$ continuously. Let $\rho_{1},\rho_{2}$ are two solutions with the same initial data, which means
$$\partial_{t}\rho +\rm{div}(\rho_{1}u_{1})=\varepsilon\triangle\rho_{1},\ \  \partial_{t}\rho +\rm{div}(\rho_{2}u_{2})=\varepsilon\triangle\rho_{2}.$$
Subtracting the above two equations, multiplying the result equation with $-\triangle(\rho_{1}-\rho_{2})$ and integrating by parts with respect to $x$ over $\mathbb{T}^{3}$, we have
\begin{equation*}
\begin{aligned}
\sup_{t\in [0,\tau]}\|\rho_{1}-\rho_{2}\|_{H^{1}}\leq \tau C( \rho_{0},\varepsilon,\|u_{1}\|_{L^{1}((0,\tau);W^{1,\infty})},\|u_{1}\|_{L^{1}((0,\tau);W^{1,\infty})})
\|u_{1}-u_{2}\|_{H^{1}},
\end{aligned}
\end{equation*}
moreover, for $u\in C([0,T];X_{n})$ is a given vector field, similarly by using the bootstrap method and compactness analysis, we can prove
\begin{equation}\label{b5}
\begin{aligned}
\|\rho_{1}&-\rho_{2}\|_{C([0,\tau];C^{7}(\mathbb{T}^{3}))}\\
&\leq \tau C( \rho_{0},\varepsilon,\|u_{1}\|_{L^{1}((0,\tau);X_{n})},\|u_{1}\|_{L^{1}((0,\tau);X_{n})})
\|u_{1}-u_{2}\|_{C([0,\tau];X_{n})}.
\end{aligned}
\end{equation}

So if we introduce the operator $\mathcal{S}:C([0,T];X_{n})\rightarrow C([0,T];C^{7}(\mathbb{T}^{3}))$ by $ \mathcal{S}(u)=\rho$, we have the following Proposition
\begin{Proposition}\label{pro1}
If $0<\underline{\rho}\leq \rho\leq \overline{\rho},\ \rho_{0}\in C^{\infty}(\mathbb{T}^{3}),\ u\in C([0,T];X_{n})$, then there exists an operator $\mathcal{S}:C([0,T];X_{n})\rightarrow C([0,T];C^{7}(\mathbb{T}^{3}))$ satisfying
\begin{itemize}
  \item $\rho=\mathcal{S}(\rho)$ is an unique solution to the problem (\ref{b2}).
  \item $0<\underline{\rho}e^{-\int_{0}^{T}\|{\rm div}u\|_{L^{\infty}}dt}\leq \rho(x,t)\leq \overline{\rho}e^{\int_{0}^{T}\|{\rm div}u\|_{L^{\infty}}dt},\ \ \forall x\in \mathbb{T}^{3},\ \ t\geq0.$
  \item $
\|\mathcal{S}(u_{1})-\mathcal{S}(u_{2})\|_{C([0,\tau];C^{7}(\mathbb{T}^{3}))}$\\
$\leq \tau C( \rho_{0},\varepsilon,\|u_{1}\|_{L^{1}((0,\tau);X_{n})},\|u_{1}\|_{L^{1}((0,T\tau);X_{n})})
\|u_{1}-u_{2}\|_{C([0,\tau];X_{n})}$, for any $\tau \in [0,T]$ and $u_{1},u_{2}\in M_{k}=\{u\in C([0,T];X_{n});\|u\|_{C([0,T];X_{n})}\leq k,\ t\in\ [0,T]\}.$
\end{itemize}
\end{Proposition}
\begin{Remark}\label{re1}
The proposition \ref{pro1} shows the operator $\mathcal{S}$ is also Lipschitz continuous for sufficient small time $t$.
\end{Remark}
\subsection{Faedo-Galerkin approximation}
Next we wish to solve the momentum equation on the space $X_{n}$ by using the Faedo-Galerkin approximation method. To this end, for given $\rho=\mathcal{S}(u)$, we are looking for a approximate solution $u_{n}\in C([0,T];X_{n})$ satisfying
\begin{equation}\label{b6}
\begin{aligned}
&\int_{\mathbb{T}^{3}}\rho u_{n}(T)\varphi dx-\int_{\mathbb{T}^{3}} m_{0}\varphi dx+\mu\int_{0}^{T}\int_{\mathbb{T}^{3}} \triangle u_{n}\cdot\triangle \varphi dxdt\\
&-\int_{0}^{T}\int_{\mathbb{T}^{3}} (\rho u_{n}\otimes u_{n})\cdot\nabla \varphi dxdt-\int_{0}^{T}\int_{\mathbb{T}^{3}} \rho \mathbb{D}u_{n}\cdot\nabla \varphi dxdt-\int_{0}^{T}\int_{\mathbb{T}^{3}} P(\rho){\rm div} \varphi dxdt\\
&+\eta\int_{0}^{T}\int_{\mathbb{T}^{3}} \rho^{-6}\rm {div}\varphi dxdt+\varepsilon\int_{0}^{T}\int_{\mathbb{T}^{3}} \nabla\rho\cdot\nabla u_{n}\varphi dxdt+r_{0}\int_{0}^{T}\int_{\mathbb{T}^{3}} u_{n}\varphi dxdt\\
&+r_{1}\int_{0}^{T}\int_{\mathbb{T}^{3}} \rho |u_{n}|u_{n} \varphi dxdt-\delta\int_{0}^{T}\int_{\mathbb{T}^{3}} \rho \nabla\triangle^{3} \rho\varphi dxdt=\int_{0}^{T}\int_{\mathbb{T}^{3}} \rho \nabla\Phi \varphi dxdt,
\end{aligned}
\end{equation}
for any test function $\varphi\in X_{n}$. The extra term $\mu \triangle^{2}u_{n}$ is not only necessary to extend the local solution obtained by the fixed point theorem to a global one at the Gerlakin level but also to make sure $\partial_{t}(\frac{\nabla\rho}{\rho})\in L^{2}((0,T);L^{2})$ so that it can be taken as a test function when we compute the B-D entropy at the next level, and the extra terms $\eta\nabla\rho^{-6}$ and $\delta\rho\nabla\triangle^{3}\rho$ are also necessary to keep the density bounded, and bounded away from below with a positive constant for all the time, this enables us to take $\frac{\nabla\rho}{\rho}$ as a test function to derive the B-D entropy, and the term $r_{0}u_{n}$ is used to control the density near the vacuum, $\rho|u_{n}|u_{n}$ is used to make sure that $\sqrt{\rho}u$ is strong convergence in $L^{2}(0,T;L^{2}(\mathbb{T}^{3}))$ at the last approximation level.

To solve (\ref{b6}), we follow the same arguments as in \cite{feireisl2004,fnp2001,jungel2010}, and introduce the following operators, giving a function $\rho\in L^{1}(\mathbb{T}^{3})$ with $\rho >\underline{\rho}>0$:
$$\mathcal{M}[\rho]:X_{n}\rightarrow X_{n}^{*},\ \ <\mathcal{M}[\rho] u,v>=\int_{\mathbb{T}^{3}}\rho u\cdot vdx,\ \ u,\ v\in X_{n}.$$
Similarly in \cite{fnp2001}, it's easy to check that the operator $M[\rho]$ satisfies the following properties:
\begin{itemize}
  \item $\|\mathcal{M}[\rho]\|_{\mathcal{L}(X_{n},X_{n}^{*})}\leq C(n)\|\rho\|_{L^{1}}.$
  \item $\|\mathcal{M}[\rho]\|_{\mathcal{L}(X_{n},X_{n}^{*})}\geq \inf_{x\in\mathbb{T}^{3}}\rho$
  \item If $\inf_{x\in\mathbb{T}^{3}}\rho\geq \underline{\rho}>0$, then the operator is invertible with
      $$\|\mathcal{M}^{-1}[\rho]\|_{\mathcal{L}(X_{n}^{*},X_{n})}\leq \underline{\rho}^{-1},$$
      where $\mathcal{L}(X_{n}^{*},X_{n})$ is the set of bounded liner mappings from $X_{n}^{*}$ to $X_{n}$.
  \item $\mathcal{M}^{-1}[\rho]$ is Lipschitz continuous in the sense
  $$\|\mathcal{M}^{-1}[\rho_{1}]-\mathcal{M}^{-1}[\rho_{2}]\|_{\mathcal{L}(X_{n}^{*},X_{n})}\leq C(n,\underline{\rho})\|\rho_{1}-\rho_{2}\|_{L^{1}(\mathbb{T}^{3})}$$
  for all $\rho_{1},\rho_{2}\in L^{1}(\mathbb{T}^{3})$ such that $\rho_{1},\rho_{2}\geq \underline{\rho}>0$.
\end{itemize}
\begin{proof}
Here, we omit the proof, for more details, we refer the readers to \cite{feireisl2004,fnp2001,jungel2010}.
\end{proof}

Then by using the operators $\mathcal{M}[\rho]$ and $\rho=\mathcal{S}(u_{n})$, we rewrite (\ref{b6}) as the following fixed-point problem
\begin{equation}\label{b7}
\begin{aligned}
u_{n}(t)=\mathcal{M}^{-1}[(\mathcal{S}(u_{n})(t)]\big(\mathcal{M}[\rho_{0}](u_{0})
+\int_{0}^{T}\mathcal{N}
(\mathcal{S}(u_{n}),u_{n})(s)ds\big),
\end{aligned}
\end{equation}
where
\begin{equation*}
\begin{aligned}
\mathcal{N}
(\mathcal{S}(u_{n}),u_{n})(s)&=\rho \nabla\Phi-{\rm div}(\rho u_{n}\otimes u_{n})+{\rm div}(\rho \mathbb{D}u_{n})-\mu \triangle^{2}u_{n}-\varepsilon\nabla \rho\cdot\nabla u_{n}\\
&-\nabla P(\rho)+\eta\nabla\rho^{-6}-r_{0}u_{n}-r_{1}\rho|u_{n}|u_{n}+\delta\rho\nabla\triangle^{3}\rho.
\end{aligned}
\end{equation*}
Thanks to the Lipschitz continuous estimates for $\mathcal{S}$ and $\mathcal{M}^{-1}$, this equation can be solved by using the fixed-point theorem of Banach for a short time $[0,T^{'}]$, where $T^{'}\leq T$, in the space $C([0,T];X_{n})$. Thus there exists a unique local-in-time solution $(\rho_{n},u_{n},\Phi(\rho_{n}))$ to (\ref{b2}) and (\ref{b6}). Next we will extend this obtained local solution to be a global one.

Differentiating $(\ref{b6})$ with respect to time $t$, taking $\phi=u_{n}$ and integrating by parts with respect to $x$ over $\mathbb{T}^{3}$, we have the following energy estimate
\begin{equation}\label{b8}
\begin{aligned}
&\frac{1}{2}\frac{d}{dt}\int\rho u_{n}^{2}+\frac{1}{2}\int \rho_{t}|u_{n}|^{2}+\int{\rm div}(\rho u_{n})|u_{n}|^{2}+\int\rho u_{n}\cdot\nabla u_{n}:u_{n}\\
&+\int\nabla P(\rho)\cdot u_{n}-\eta\int\nabla \rho^{-6}\cdot u_{n}+\varepsilon\int \nabla\rho\cdot \nabla u_{n}\cdot u_{n}+\int \rho |\mathbb{D}u_{n}|^{2}+r_{0}\int|u_{n}|^{2}\\
&+r_{1}\int \rho|u_{n}|^{3}+\delta\int {\rm div}(\rho u_{n})\triangle^{3}\rho+\mu\int |\triangle u_{n}|^{2}=-\int {\rm div}(\rho u_{n})\Phi,
\end{aligned}
\end{equation}
firstly, we estimate the terms on the left hand side one by one:
\begin{equation}\label{b9}
\begin{aligned}
&\frac{1}{2}\int \rho_{t}|u_{n}|^{2}+\int{\rm div}(\rho u)|u_{n}|^{2}+\int\rho u_{n}\cdot\nabla u_{n}:u_{n}\\
&=\frac{1}{2}\int [\varepsilon\triangle \rho-{\rm div}(\rho u_{n})]|u_{n}|^{2}+\int{\rm div}(\rho u)|u_{n}|^{2}+\int \rho u_{n}^{i}\partial_{i}u_{n}^{j}u_{n}^{j}\\
&=\frac{1}{2}\int{\rm div}(\rho u_{n})|u_{n}|^{2}-\varepsilon\int \nabla\rho\cdot \nabla u_{n}\cdot u_{n}-\frac{1}{2}\int{\rm div}(\rho u_{n})|u_{n}|^{2}\\
&=-\varepsilon\int \nabla\rho\cdot \nabla u_{n}\cdot
u_{n},
\end{aligned}
\end{equation}
where we used the approximate mass equation (\ref{b2}).
\begin{equation}\label{b10}
\begin{aligned}
\int \nabla P(\rho) \cdot u_{n}&=\int \nabla\int_{1}^{\rho}\frac{P^{'}(s)}{s}ds\  (\rho u_{n})dx=-\int \int_{1}^{\rho}\frac{P^{'}(s)}{s}ds[\varepsilon\triangle \rho-\rho_{t}]dx\\
&=\frac{d}{dt}\int \Pi(\rho)dx+\varepsilon\int \frac{P^{'}(\rho)}{\rho}|\nabla \rho|^{2},
\end{aligned}
\end{equation}
where we used (\ref{b2}) and integration by parts, and $\Pi(\rho)=\rho\int_{1}^{\rho}\frac{P(s)}{s^{2}}ds$.
Next we will deal with the cold pressure and high order derivative of the density terms as follows
\begin{equation}\label{b11}
\begin{aligned}
-\eta\int \nabla\rho^{-6}\cdot u_{n}&=-\frac{6}{7}\eta\int \rho u_{n}\cdot\nabla \rho^{-7}=\frac{6}{7}\eta\int \rho^{-7}[\varepsilon\triangle\rho-\rho_{t}]\\
&=\frac{1}{7}\eta\frac{d}{dt}\int\rho^{-6}+\frac{2}{3}
\eta\varepsilon\int |\nabla\rho^{-3}|^{2},
\end{aligned}
\end{equation}
\begin{equation}\label{b12}
\begin{aligned}
\delta\int {\rm div}(\rho u_{n})\triangle^{3}\rho&=\delta\int [\varepsilon\triangle\rho-\rho_{t}]\triangle^{3}\rho\\
&=\frac{\delta}{2}\frac{d}{dt}\int|\nabla
\triangle\rho|^{2}+\delta\varepsilon\int|\triangle^{2}\rho|^{2},
\end{aligned}
\end{equation}
finally, we will estimate the poisson term on the right hand side
\begin{equation}\label{b13}
\begin{aligned}
-\int{\rm div}(\rho u_{n})\Phi&=-\int [\varepsilon\triangle\rho-\rho_{t}]\Phi=\int \partial_{t}(\rho -1)\Phi-\varepsilon\int \triangle (\rho-1)\Phi\\
&=-\int\frac{1}{4\pi G}\partial_{t}\triangle\Phi\cdot \Phi-\varepsilon\int (\rho-1)\triangle \Phi\\
&=\frac{1}{8\pi G}\frac{d}{dt}\int |\nabla\Phi|^{2}+\varepsilon\int 4\pi G(\rho -1)^{2},
\end{aligned}
\end{equation}
where we used  the equation $(\ref{a1})_{3}$.

Then substituting (\ref{b9})-(\ref{b13}) into (\ref{b8}) and integrating the result equation with respect to $t$ over $[0,T]$, yields
\begin{equation}\label{b14}
\begin{aligned}
&E(t)+\varepsilon\int_{0}^{T}\int\frac{P^{'}(\rho)}{\rho}|\nabla \rho|^{2}+\frac{2}{3}\eta\varepsilon\int_{0}^{T}\int|\nabla \rho^{-3}|^{2}+\int_{0}^{T}\int\rho |\mathbb{D}u_{n}|^{2}\\
&\ \ +r_{0}\int_{0}^{T}\int|u_{n}|^{2}+r_{1}\int_{0}^{T}\int\rho|u_{n}|^{3}
+\mu\int_{0}^{T}\int|\triangle u_{n}|^{2}+\delta\varepsilon\int_{0}^{T}\int|\triangle^{2}\rho|^{2}\\
&=4\pi G\varepsilon\int_{0}^{T}\int(\rho-1)^{2}+E_{0},
\end{aligned}
\end{equation}
where
 $$E(t)=\int_{\mathbb{T}^{3}}( \frac{1}{2}\rho u_{n}^{2}+{\Pi}(\rho)+\frac{\eta}{7}\rho^{-6}-\frac{1}{8\pi G}|\nabla \Phi(\rho)|^{2}+\frac{\delta}{2}|\nabla\triangle\rho|^{2})dx, \ {\Pi}(\rho)=\rho\int_{1}^{\rho}\frac{P(s)}{s^{2}}ds$$
and
$$E_{0}=\int_{\mathbb{T}^{3}}( \frac{1}{2}\rho_{0} u_{n}^{2}+\Pi(\rho_{0})+\frac{\eta}{7}\rho_{0}^{-6}-\frac{1}{8\pi G}|\nabla \Phi(\rho_{0})|^{2}+\frac{\delta}{2}|\nabla\triangle\rho_{0}|^{2})dx.$$
Moreover, because of (\ref{a3}), we have
\begin{equation}\label{b15}
\begin{aligned}
\Pi(\rho)=\rho\int_{1}^{\rho}\frac{P(s)}{s^{2}}ds\geq \rho\int_{1}^{\rho}\frac{\frac{1}{a\gamma}s^{\gamma}-bs}{s^{2}}ds=\frac{1}{a\gamma(\gamma-1)}
[\rho^{\gamma}-\rho]-b\rho\log\rho,
\end{aligned}
\end{equation}
and
\begin{equation}\label{b16}
\begin{aligned}
\varepsilon\int_{0}^{T}\int\frac{P^{'}(\rho)}{\rho}|\nabla\rho|^{2}&\geq \varepsilon\int_{0}^{T}\int\frac{\frac{1}{a}\rho^{\gamma-1}-b}{\rho}|\nabla\rho|^{2}\\
&=\frac{4\varepsilon}
{a\gamma^{2}}\int_{0}^{T}\int|\nabla\rho^{\frac{\gamma}{2}}|^{2}-b\varepsilon\int_{0}^{T}\int
\frac{1}{\rho}|\nabla \rho|^{2},
\end{aligned}
\end{equation}
furthermore, if $\gamma>\frac{4}{3}$, we also have
\begin{equation}\label{b17}
\begin{aligned}
\frac{1}{8\pi G}\int |\nabla\Phi|^{2}&\leq C\|\nabla\Phi\|_{L^{2}}^{2}\leq C\|\nabla^{2}\Phi\|_{L^{\frac{6}{5}}}^{2}\leq C\|\rho-1\|_{L^{\frac{6}{5}}}^{2}\\
&\leq C(1+\|\rho\|_{L^{\frac{6}{5}}}^{2})\leq C(1+\|\rho\|_{L^{1}}^{\frac{5\gamma-6}{3(\gamma-1)}}\|\rho\|_{L^{\gamma}}^{\frac{\gamma}{3(\gamma
-1)}})\leq C+\varsigma\|\rho\|_{L^{\gamma}}^{\gamma},
\end{aligned}
\end{equation}
where $0<\varsigma\ll 1$ is a fixed constant, $C$ is a generic positive constant independent of $\varepsilon,\eta,\delta, r_{0}$, and we also used conservation of mass, Sobolev inequality, Young inequality.

Then substituting (\ref{b15})-(\ref{b17}) into (\ref{b14}), we have
\begin{equation}\label{b18}
\begin{aligned}
&\int\frac{1}{2}\rho u_{n}^{2}+\frac{1}{a\gamma(\gamma-1)}\rho^{\gamma}+\frac{\eta}{7}\rho^{-6}+\frac{\delta}{2}|\nabla
\triangle\rho|^{2}dx+\frac{4\varepsilon}{a\gamma^{2}}\int_{0}^{T}\int|\nabla\rho^{\frac{\gamma}{2}}
|^{2}\\
&+\frac{2}{3}\eta\varepsilon\int_{0}^{T}\int|\nabla \rho^{-3}|^{2}+\int_{0}^{T}\int\rho |\mathbb{D}u_{n}|^{2}+r_{0}\int_{0}^{T}\int|u_{n}|^{2}+r_{1}\int_{0}^{T}\int\rho|u_{n}|^{3}
\\
&+\mu\int_{0}^{T}\int|\triangle u_{n}|^{2}+\delta\varepsilon\int_{0}^{T}\int|\triangle^{2}\rho|^{2}\\
&=4\pi G\varepsilon\int_{0}^{T}\int(\rho-1)^{2}+\int\frac{1}{\a\gamma(\gamma-1)}\rho+b\int\rho\log\rho
+C+\varsigma\|\rho\|_{L^{\gamma}}^{\gamma}\\
&+b\varepsilon\int\frac{1}{\rho}|\nabla\rho|^{2}\leq \varsigma^{'}\|\rho\|_{L^{\gamma}}^{\gamma}+C+C\varepsilon\int_{0}^{T}\int \rho^{2}+b\varepsilon\int\frac{1}{\rho}
|\nabla \rho|^{2},
\end{aligned}
\end{equation}
where $\varsigma^{'}$ is a  sufficient small positive constant, $C$ is a generic positive constant only depending on the initial data and $T$.

Because
\begin{equation}\label{b19}
\begin{aligned}
C\varepsilon\int_{0}^{T}\int \rho^{2}\leq C\varepsilon\int_{0}^{T}\|\rho\|_{L^{1}}^{\frac{4}{3}}\|\nabla^{3}\rho\|_{L^{2}}^{\frac{2}{3}}dt\leq C+C\frac{\varepsilon}{\delta}\int_{0}^{T}\delta\|\nabla^{3}\rho\|_{L^{2}}^{2}dt
\end{aligned}
\end{equation}
and
\begin{equation}\label{b20}
\begin{aligned}
b\varepsilon&\int_{0}^{T}\int\frac{1}{\rho}|\nabla\rho|^{2}\leq C\varepsilon\int_{0}^{T}\|\rho^{-1}\|_{L^{6}}\|(\nabla\rho)^{2}\|_{L^{\frac{6}{5}}}dt\\
&\leq C\varepsilon\int_{0}^{T}\|\rho^{-1}\|_{L^{6}}\|\rho\|_{L^{1}}^{\frac{11}{9}}\|\nabla^{3}\rho\|
_{L^{2}}^{\frac{7}{9}}dt\\
&\leq C+\frac{C\varepsilon}{\eta}\int_{0}^{T}\eta\|\rho^{-1}\|_{L^{6}}^{6}dt+\frac{C\varepsilon}{\delta}
\int_{0}^{T}\delta\|\nabla^{3}\rho\|_{L^{2}}^{2}dt,
\end{aligned}
\end{equation}
then substituting (\ref{b19})-(\ref{b20}) into (\ref{b18}) and using the Gronwall inequality gives
\begin{equation}\label{b21}
\begin{aligned}
\int\eta\rho^{-6}+\delta|\nabla\triangle\rho|^{2}dx\leq C+C(\frac{\varepsilon}{\delta}+\frac{\varepsilon}{\eta})Te^{C(\frac{\varepsilon}{\delta}
+\frac{\varepsilon}{\eta})T}.
\end{aligned}
\end{equation}
In combination with (\ref{b18}) and (\ref{b21}), we have the following energy inequality
\begin{equation}\label{b22}
\begin{aligned}
&\int\frac{1}{2}\rho u_{n}^{2}+\frac{1}{a\gamma(\gamma-1)}\rho^{\gamma}+\frac{\eta}{7}\rho^{-6}+\frac{\delta}{2}|\nabla
\triangle\rho|^{2}dx+\frac{4\varepsilon}{a\gamma^{2}}\int_{0}^{T}\int|\nabla\rho^{\frac{\gamma}{2}}
|^{2}\\
&\ \ +\frac{2}{3}\eta\varepsilon\int_{0}^{T}\int|\nabla \rho^{-3}|^{2}+\int_{0}^{T}\int\rho |\mathbb{D}u_{n}|^{2}+r_{0}\int_{0}^{T}\int|u_{n}|^{2}+r_{1}\int_{0}^{T}\int\rho|u_{n}|^{3}
\\
&\ \ \ \ \ \ \ \ \ \ +\mu\int_{0}^{T}\int|\triangle u_{n}|^{2}+\delta\varepsilon\int_{0}^{T}\int|\triangle^{2}\rho|^{2}\\
&\leq C+C(\frac{\varepsilon}{\delta}+\frac{\varepsilon}{\eta})[1+(\frac{\varepsilon}{\delta}
+\frac{\varepsilon}{\eta})Te^{C(\frac{\varepsilon}{\delta}
+\frac{\varepsilon}{\eta})T}]T
\end{aligned}
\end{equation}
So the energy inequality (\ref{b22}) yields
$$\int_{0}^{T}\|\triangle u_{n}\|_{L^{2}}^{2}dt\leq C(\varepsilon,\eta,\delta)<+\infty,$$
where $C(\varepsilon,\eta,\delta)$ denotes a positive constant especially depending on $\varepsilon,\eta,\delta$ but independent of $n$, and due to dim$X_{n}<\infty$ and (\ref{b4}), then the density is bounded and bounded away from blow with a positive constant, which means there exists a constant $c>0$ such that
\begin{equation}\label{b23}
0<\frac{1}{c}\leq \rho_{n}\leq c,
\end{equation}
for all $t\in [0,T^{*})$. Furthermore, the energy inequality also gives us
\begin{equation}\label{b24}
\sup_{t\in(0,T^{*})}\int\rho_{n}u_{n}^{2}\leq C<\infty,
\end{equation}
which together with (\ref{b23}) and (\ref{b24}), implies
$$\sup_{t\in(0,T^{*})}\int\|u_{n}\|_{L^{\infty}}\leq C<\infty,$$
where we used the fact that all the norms are equivalence on $X_{n}$. Then we can repeat above argument many times and use the compactness analysis, we can obtain $u_{n}\in C([0,T];X_{n})$, so we can extend $T^{*}$ to $T$. Thus there exists a global solution $(\rho_{n},u_{n},\Phi(\rho_{n}))$ to (\ref{b2}), (\ref{b6}) for any time $T$.

To conclude this part, we have the following proposition on the approximate solutions $(\rho_{n},u_{n},\Phi(\rho_{n}))$:
\begin{Proposition}\label{pro2}
Let $(\rho_{n},u_{n},\Phi(\rho_{n}))$ be the solutions of (\ref{b2}), (\ref{b6}) on $(0,T)\times \mathbb{T}^{3}$ constructed above, then the solutions must satisfy the energy inequality (\ref{b22}). In particular, we have the following estimates
\begin{equation}\label{b25}
\begin{aligned}
&\sqrt{\rho_{n}}u_{n}\in L^{\infty}(0,T;L^{2}),\sqrt{\rho}\mathbb{D}u_{n}\in L^{2}(0,T;L^{2}),\sqrt{\mu}\triangle u_{n}\in L^{2}(0,T;L^{2}),\\
&\rho_{n}\in L^{\infty}(0,T;L^{1}\cap L^{\gamma}),\sqrt{\varepsilon}\nabla \rho_{n}^{\frac{\gamma}{2}}\in L^{2}(0,T;L^{2}),\eta\rho_{n}^{-6}\in L^{\infty}(0,T;L^{1}),\\
&\sqrt{\varepsilon\eta}\nabla\rho_{n}^{-3}\in L^{2}(0,T;L^{2}),\sqrt{\delta}\rho_{n}\in L^{\infty}(0,T;H^{3}),\sqrt{\delta\varepsilon}\rho_{n}\in L^{2}(0,T;H^{4}),\\
&\sqrt{r_{0}}u_{n}\in L^{2}(0,T;L^{2}),\rho_{n}^{\frac{1}{3}}u_{n}\in L^{3}(0,T;L^{3}).
\end{aligned}
\end{equation}

\end{Proposition}
\subsection{Passing to the limits as $n\rightarrow\infty$.}We perform first the limit as $n\rightarrow\infty$, $\varepsilon,\eta,\delta,r_{0}>0$ being fixed. Based on the above estimates which are uniform on $n$ and Aubin-Lions Lemma, we have the following compactness results.
\subsubsection{Step1.Convergence of $\rho_{n}$, Pressure $P(\rho_{n})$ and gravitational force $\nabla\Phi(\rho_{n})$.}
\begin{Lemma}\label{l1}
The following estimates hold for any fixed positive constants $\varepsilon,\eta,\delta$ and $r_{0}$:
\begin{equation}\label{b26}
\begin{aligned}
&\|\rho_{n}\|_{L^{\infty}(0,T;H^{3})}+\|\rho_{n}\|_{L^{2}(0,T;H^{4})}\leq K,\ \|\partial_{t}\rho_{n}\|_{L^{2}(0,T;H^{-1})}\leq K,\\
&\|\rho_{n}^{\gamma}\|_{L^{\frac{5}{3}}(0,T;L^{\frac{5}{3}})}\leq K,\ \ \|\rho_{n}^{-6}\|_{L^{\frac{5}{3}}(0,T;L^{\frac{5}{3}})}\leq K,\\
&\|\nabla\Phi(\rho_{n})\|_{C([0,T];L^{2})}\leq C\|\rho-1\|_{C([0,T];L^{\frac{6}{5}})}\leq K
\end{aligned}
\end{equation}
where $K$ is independent of $n$, depends on $\varepsilon,\eta,\delta$,$r_{0}$, initial data and $T$. Moreover, up to an extracted subsequence
\begin{equation}\label{b27}
\begin{aligned}
&\rho_{n}\rightarrow\rho\ \  a.e. \ \ and \ \ strongly\ in\  C([0,T];H^{3}) \ \ ,\\
&P(\rho_{n})\rightarrow P(\rho)\ \  a.e. \ \ and\ \  strongly\ in\  L^{1}(0,T;L^{1})\ \  ,\\
&\rho_{n}^{-6}\rightarrow \rho^{-6}\ \  a.e. \ \ and\ \  strongly\ in\  L^{1}(0,T;L^{1})\ \  ,\\
&\nabla\Phi(\rho_{n})\rightarrow \nabla\Phi(\rho)\ \ strongly\  in\  L^{2}(0,T;L^{2})\ \  .
\end{aligned}
\end{equation}
\end{Lemma}
\begin{proof}
By (\ref{b2}), we have
\begin{equation*}
\begin{aligned}
\int_{0}^{T}\int(\rho_{n})_{t}\varphi&=-\varepsilon\int_{0}^{T}\int\nabla \rho_{n}\nabla\varphi+\int_{0}^{T}\int(\rho_{n}u_{n})\nabla\varphi\\
&\leq C\|\nabla\rho_{n}\|_{L^{2}(0,T;L^{2})}\|\nabla\varphi\|_{L^{2}(0,T;L^{2})}\\
&\ \ \ \ \ +C\|\rho_{n}\|
_{L^{\infty}(0,T;L^{\infty})}\|u_{n}\|_{L^{2}(0,T;L^{2})}\|\nabla\varphi\|_{L^{2}(0,T;L^{2})}
\leq C,
\end{aligned}
\end{equation*}
holds for any $\varphi\in {L^{2}(0,T;H^{1})}$, which yields $\partial_{t}\rho_{n}\in {L^{2}(0,T;H^{-1})}$. This together with $\rho_{n}\in L^{\infty}(0,T;H^{3})\cap L^{2}(0,T;H^{4})$, using the Aubin-Lions Lemma, we can claim $\rho_{n}\in C([0,T];H^{3})$, so up to a subsequence, we have\\
$$\rho_{n}\rightarrow\rho\ \ \ strongly\ \ \ in\ C([0,T];H^{3}), \ \ hence,\ \  \rho_{n}\rightarrow\rho\ \  a.e.$$

Next we claim that $\rho_{n}^{\gamma}$ is bounded in $L^{\frac{5}{3}}(0,T;L^{\frac{5}{3}})$.

Notice that $\nabla\rho_{n}^{\frac{\gamma}{2}}$ is bounded in ${L^{2}(0,T;L^{2})}$, using the Sobolev embedding theorem gives us $\rho_{n}^{\gamma}$ is bounded in ${L^{1}(0,T;L^{3})}$, then we apply H$\ddot{o}$lder inequality to have
$$\|\rho_{n}^{\gamma}\|_{L^{\frac{5}{3}}(0,T;L^{\frac{5}{3}})}\leq \|\rho_{n}^{\gamma}\|_{L^{\infty}(0,T;L^{1})}^{\frac{2}{5}}\|\rho_{n}^{\gamma}\|_{L^{1}(0,T;L^{3})}
^{\frac{3}{5}}\leq K.$$
Similarly, we can show $\rho_{n}^{-6}$ is bounded in $L^{\frac{5}{3}}(0,T;L^{\frac{5}{3}})$ too. Moreover, for $\rho_{n}\rightarrow \rho$ a.e., so $\rho_{n}^{\gamma}\rightarrow \rho^{\gamma}$ a.e..
Recall that the pressure satisfies $\frac{1}{a}\rho^{\gamma-1}-b\leq P^{'}\leq a\rho^{\gamma-1}+b$ and $P\in C^{1}(\mathbb{R}^{+})$, integrating this inequality we have
$$|P(\rho_{n})|\leq C(\rho_{n}^{\gamma}+\rho_{n}),$$
it implies that $P(\rho_{n})$ is bounded in $L^{\frac{5}{3}}(0,T;L^{\frac{5}{3}})$ due to $\rho_{n}^{\gamma}$ is bounded in $L^{\frac{5}{3}}(0,T;L^{\frac{5}{3}})$. For $\rho_{n}^{\gamma}\rightarrow \rho^{\gamma}$ a.e., using the Egoroffs theorem, we have
$$ P(\rho_{n})\rightarrow P(\rho)\ \ \ strongly\ \ \ in\ L^{1}(0,T;L^{1}).$$
Next, we show that the density is bounded away from zero with a positive constant for all the time $t\in [0,T]$ by using the Sobolve inequality.

For \begin{equation}\label{b28}
\begin{aligned}
\|\rho_{n}^{-1}\|_{L^{\infty}}\leq \|\rho_{n}^{-1}\|_{L^{6}}^{\frac{1}{2}}\|\nabla^{2}\rho_{n}^{-1}\|_{L^{2}}^{\frac{1}{2}},
\end{aligned}
\end{equation}
\begin{equation}\label{b29}
\begin{aligned}
\|\nabla^{2}\rho_{n}^{-1}\|_{L^{2}}&\leq C(\|\rho_{n}^{-2}\nabla^{2}\rho_{n}\|_{L^{2}}+\|\rho_{n}^{-3}(\nabla\rho_{n})^{2}\|_{L^{2}}\\
&\leq C\|\rho_{n}^{-1}\|_{L^{6}}^{2}(\|\rho_{n}\|_{L^{1}}+\|\nabla^{3}\rho_{n}\|_{L^{2}})+C\|\rho_{n}
^{-1}\|_{L^{6}}^{3}\|\nabla\rho_{n}\|_{L^{\infty}}^{2}\\
&\leq C(1+\|\rho_{n}^{-1}\|_{L^{6}}^{3})(1+\|\nabla^{3}\rho_{n}\|_{L^{2}}^{2}),
\end{aligned}
\end{equation}
substituting (\ref{b29}) into (\ref{b28}), yields
\begin{equation}\label{b30}
\begin{aligned}
\|\rho_{n}^{-1}\|_{L^{\infty}}&\leq C\|\rho_{n}^{-1}\|_{L^{6}}^{\frac{1}{2}}(1+\|\rho_{n}^{-1}\|_{L^{6}})^{\frac{3}{2}}
(1+\|\nabla^{3}\rho_{n}\|_{L^{2}})\\
&\leq C(1+\|\rho_{n}^{-1}\|_{L^{6}})^{2}
(1+\|\nabla^{3}\rho_{n}\|_{L^{2}})\leq C(\eta,\delta,T),
\end{aligned}
\end{equation}
where here the constant $C(\eta,\delta,T)$ depends on $\eta,\delta$ and $T$ but independent of $n$.

So immediately, we have $\frac{1}{\rho_{n}}\rightarrow\frac{1}{\rho}$ a.e., furthermore we show $\rho_{n}^{-6}\rightarrow \rho^{-6}$  a.e., together with $\rho_{n}^{-6}\in L^{\frac{5}{3}}(0,T;L^{\frac{5}{3}})$ and Egoroffs theorem, we have
$$\rho_{n}^{-6}\rightarrow \rho^{-6},\ \ \ strongly\ \ \ in\ L^{1}(0,T;L^{1}).$$

By using the G-N inequality, yields
$$\|\nabla\Phi(\rho_{n})\|_{L^{2}}\leq C\|\nabla^{2}\Phi(\rho_{n})\|_{L^{\frac{6}{5}}}\leq C\|\rho_{n}-1\|_{L^{\frac{6}{5}}},$$
because $\rho_{n}$ convergence to  $\rho$ strongly in $C([0,T];H^{3})$, so $$\nabla\Phi(\rho_{n})\rightarrow \nabla\Phi(\rho) \ \ strongly\ \  in\  L^{2}(0,T;L^{2}).$$

Then the proof of this Lemma is completed.
\end{proof}
\subsubsection{Step2. Convergence of momentum}
\begin{Lemma}\label{l2}
Up to an extracted subsequence
$$\rho_{n}u_{n}\rightarrow\rho u\ \ \ \ a.e.\ \ and \ strongly\ \ in \ \ L^{2}(0,T;L^{2}).$$

\end{Lemma}
\begin{proof}
From the energy estimates, we know that $u_{n}$ is bounded in $L^{2}(0,T;L^{2})$, so up to a subsequence, we have
$$u_{n}\rightharpoonup u\ \ \ \ in\ \ L^{2}(0,T;L^{2}),$$
recall that $\rho_{n}\rightarrow \rho$ strongly in $C([0,T];H^{3})$, so we have
$$\rho_{n}u_{n}\rightharpoonup \rho u\ \ \ \ weakly\ \ in\ \ L^{1}(0,T;L^{1}).$$
Moreover, since $\rho_{n}\in L^{\infty}(0,T;H^{3}), u_{n}\in L^{2}(0,T;H^{2})$, we can show $$\nabla(\rho_{n}u_{n})=\nabla\rho_{n}u_{n}+\rho_{n}\nabla u_{n}\in L^{2}(0,T;L^{2}),$$
together with $\rho_{n}u_{n}\in L^{2}(0,T;L^{2})$, we have $\rho_{n}u_{n}\in L^{2}(0,T;H^{1})$.
Next in order to use the Aubin-Lions Lemma, we only need to prove
$$\partial _{t}(\rho_{n}u_{n})\in L^{2}(0,T;H^{-s}),\ \ \ for\ some\ s>0.$$
Since,
\begin{equation}\label{b31}
\begin{aligned}
\partial_{t}(\rho_{n}u_{n})&=-{\rm div }(\rho_{n}u_{n}\otimes u_{n})-\nabla P(\rho_{n})+\eta\nabla \rho_{n}^{-6}-\mu\triangle^{2}u_{n}+{\rm div}(\rho_{n}\mathbb{D}u_{n})\\
&-r_{0}u_{n}-r_{1}\rho_{n}|u_{n}|u_{n}-\varepsilon\nabla \rho_{n}\cdot\nabla u_{n}+\delta\rho_{n}\nabla\triangle^{3}\rho_{n}+\rho_{n}\nabla\Phi,
\end{aligned}
\end{equation}
based on the energy estimates, it' s easy to check that $\partial_{t}(\rho_{n}u_{n})\in L^{2}(0,T;H^{-3})$, then by using the Aubin-Lions Lemma, we can show
$$\rho_{n}u_{n} \rightarrow g \ \ strongly\ \ in\ L^{2}(0,T;L^{2}),\  for\  some\  function\  g\in L^{2}(0,T;L^{2}),$$
moreover, since $\rho_{n}u_{n}\rightharpoonup \rho u\ \ \ \ weakly\ \ in\ \ L^{1}(0,T;L^{1}),$ so we have $$\rho_{n}u_{n} \rightarrow \rho u\ \ strongly\ \ in\ L^{2}(0,T;L^{2}).$$
 Thus the proof of this Lemma is completed.

\end{proof}
\subsubsection{Step.3 Convergence of nonlinear diffusion terms.}
Let $\varphi \in C^{\infty}_{per}([0,T];\mathbb{T}^{3})$, with  $C^{\infty}_{per}([0,T];\mathbb{T}^{3})$ defined by
$$ C^{\infty}_{per}([0,T];\mathbb{T}^{3})=\{\phi\in C^{\infty}([0,T];\mathbb{T}^{3})|\ \phi \ is\  periodic\  in\  x\}.$$
\begin{equation}\label{b32}
\begin{aligned}
&\int_{0}^{T}\int {\rm div}(\rho_{n}\mathbb{D}u_{n})\varphi=\int_{0}^{T}\int \partial_{i}(\rho_{n}(\frac{\partial_{i}u_{n}^{j}+\partial_{j}u_{n}^{i}}{2}))\varphi\\
&=-\frac{1}{2}\int_{0}^{T}\int \rho_{n}\partial_{i}u_{n}^{j}\partial_{i}\varphi-\frac{1}{2}\int_{0}^{T}\int \rho_{n}\partial_{j}u_{n}^{i}\partial_{i}\varphi\\
&=-\frac{1}{2}\int_{0}^{T}\int \partial_{i}(\rho_{n}u_{n}^{j})\partial_{i}\varphi+\frac{1}{2}\int_{0}^{T}\int\partial_{i}
\rho_{n}u_{n}^{j}\partial_{i}\varphi-\frac{1}{2}\int_{0}^{T}\int \partial_{j}(\rho_{n}u_{n}^{i})\partial_{i}\varphi\\
&\ \ \ \ +\frac{1}{2}\int_{0}^{T}\int\partial_{j}
\rho_{n}u_{n}^{i}\partial_{i}\varphi\\
&=\frac{1}{2}\int_{0}^{T}\int (\rho_{n}u_{n}^{j})\partial_{ii}\varphi+\frac{1}{2}\int_{0}^{T}\int\partial_{i}
\rho_{n}u_{n}^{j}\partial_{i}\varphi+\frac{1}{2}\int_{0}^{T}\int (\rho_{n}u_{n}^{i})\partial_{ij}\varphi\\
&\ \ \ \ +\frac{1}{2}\int_{0}^{T}\int\partial_{j}
\rho_{n}u_{n}^{i}\partial_{i}\varphi,
\end{aligned}
\end{equation}
since $\rho_{n} \rightarrow \rho\ \ strongly\ \ in\ C([0,T];H^{3}),\rho_{n}u_{n} \rightarrow \rho u\ \ strongly\ \ in\ L^{2}(0,T;L^{2}), u_{n}\rightharpoonup u\ \ weakly\ \ in\ L^{2}(0,T;L^{2})$, so we have
$$\int_{0}^{T}\int (\rho_{n}u_{n}^{j})\partial_{ii}\varphi\rightarrow \int_{0}^{T}\int (\rho u^{j})\partial_{ii}\varphi,\ \ \int_{0}^{T}\int\partial_{i}
\rho_{n}u_{n}^{j}\partial_{i}\varphi\rightarrow \int_{0}^{T}\int\partial_{i}
\rho u^{j}\partial_{i}\varphi,$$
$$\int_{0}^{T}\int (\rho_{n}u_{n}^{i})\partial_{ij}\varphi \rightarrow \int_{0}^{T}\int (\rho u^{i})\partial_{ij}\varphi,\ \ \int_{0}^{T}\int\partial_{j}
\rho_{n}u_{n}^{i}\partial_{i}\varphi\rightarrow \int_{0}^{T}\int\partial_{j}
\rho u^{i}\partial_{i}\varphi.$$
Similarly, we have
\begin{equation}\label{b33}
\begin{aligned}
\int_{0}^{T}\int\rho_{n}\nabla \triangle^{3}\rho_{n}\varphi=-\int_{0}^{T}\int(\rho_{n}{\rm div}\varphi+\varphi\cdot\nabla \rho_{n})\triangle^{3}\rho_{n}\\
=-\int_{0}^{T}\int\triangle(\rho_{n}{\rm div}\varphi+\varphi\cdot\nabla \rho_{n})\triangle^{2}\rho_{n},
\end{aligned}
\end{equation}
we focus on the most difficult term $-\int_{0}^{T}\int\varphi\cdot\nabla\triangle \rho_{n}\triangle^{2}\rho_{n}$,
because $\rho_{n}\rightarrow \rho$ strongly in $C([0,T];H^{3})$ and $\rho_{n} \rightharpoonup \rho$ in $ L^{2}(0,T;H^{4})$, we have
$$-\int_{0}^{T}\int\varphi\cdot\nabla\triangle \rho_{n}\triangle^{2}\rho_{n} \rightarrow -\int_{0}^{T}\int\varphi\cdot\nabla\triangle \rho\triangle^{2}\rho.$$
And we can apply the above arguments to handle the other terms from
$$-\int_{0}^{T}\int\triangle(\rho_{n}{\rm div}\varphi+\varphi\cdot\nabla \rho_{n})\triangle^{2}\rho_{n}.$$
Thus we have
$$\int_{0}^{T}\int\rho_{n}\nabla \triangle^{3}\rho_{n}\varphi\rightarrow  \int_{0}^{T}\int\rho\nabla \triangle^{3}\rho\varphi,\ \ as \ \ n\rightarrow\infty.$$

With the above compactness results in hand, we are ready to pass to the limits as $n\rightarrow\infty$ in the approximate system (\ref{b2}), (\ref{b6}). Thus, we can show that $(\rho,u,\Phi)$ solves
$$\rho_{t}+{\rm div}(\rho u)=\varepsilon\triangle\rho,\ \ pointwise\ \ on\ (0,T)\times\mathbb{T}^{3},$$
$$\triangle\Phi=-4\pi G(\rho-1),\ \ pointwise\ \ on\ (0,T)\times\mathbb{T}^{3},$$
and for any test function $\varphi$, the following holds
\begin{equation}\label{34}
\begin{aligned}
&\int\rho u(T)\varphi-\int m_{0}\varphi+\mu\int_{0}^{T}\int\triangle u\cdot \triangle \varphi-\int_{0}^{T}\int(\rho u\otimes u)\cdot\nabla\varphi\\
&+\int_{0}^{T}\int\rho \mathbb{D}u\cdot\nabla \varphi-\int_{0}^{T}\int P(\rho){\rm div}\varphi+\eta\int_{0}^{T}\int \rho ^{-6}{\rm div}\varphi+\varepsilon\int_{0}^{T}\int\nabla\rho\cdot\nabla u\varphi\\
&+r_{0}\int_{0}^{T}\int u\varphi+r_{1}\int_{0}^{T}\int\rho|u|u\varphi-\delta\int_{0}^{T}\int\rho\nabla\triangle^{3}\rho
\varphi=\int_{0}^{T}\int\rho\Phi\varphi
\end{aligned}
\end{equation}

Thanks to the lower semicontinuity of norms, we can pass to the limits in the energy estimate (\ref{b22}), and we have the following energy inequality in the sense of distributions on $(0,T)$.
\begin{equation}\label{b35}
\begin{aligned}
&\sup_{t\in (0,T)}E(t)+\frac{4\varepsilon}{a\gamma^{2}}\int_{0}^{T}\int|\nabla\rho^{\frac{\gamma}{2}}
|^{2}+\frac{2}{3}\eta\varepsilon\int_{0}^{T}\int|\nabla \rho^{-3}|^{2}+\int_{0}^{T}\int\rho |\mathbb{D}u|^{2}\\
&+r_{0}\int_{0}^{T}\int|u|^{2}+r_{1}\int_{0}^{T}\int\rho|u|^{3}
+\mu\int_{0}^{T}\int|\triangle u|^{2}+\delta\varepsilon\int_{0}^{T}\int|\triangle^{2}\rho|^{2}\\
&\leq C+C_{\varepsilon},
\end{aligned}
\end{equation}
where
$$E(t)=\int\frac{1}{2}\rho u^{2}+\frac{1}{a\gamma(\gamma-1)}\rho^{\gamma}+\frac{\eta}{7}\rho^{-6}+\frac{\delta}{2}|\nabla
\triangle\rho|^{2}dx,$$
and
$$C_{\varepsilon}=C(\frac{\varepsilon}{\delta}+\frac{\varepsilon}{\eta})[1+(\frac{\varepsilon}{\delta}
+\frac{\varepsilon}{\eta})Te^{C(\frac{\varepsilon}{\delta}
+\frac{\varepsilon}{\eta})T}]T.$$

Thus, we have the following proposition on the existence of weak solutions at this level approximate system.
\begin{Proposition}\label{pro3}
There exists a weak solution to the following system
\begin{equation}\label{b36}
\begin{aligned}
\left\{\begin{array}{lll}
&\rho_{t}+{\rm div}(\rho u)=\varepsilon\triangle\rho,\\
&(\rho u)_{t}+{\rm div}(\rho u\otimes u)+\nabla P(\rho)-\eta\nabla\rho^{-6}-{\rm div}(\rho \mathbb{D}u)+\mu \triangle^{2} u+\varepsilon\nabla\rho\cdot \nabla u\\
&\ \ \ \ \ \ \ \ \ \ \ \ \ \ \ \ \ \ \ \ \ \ \ \ r_{0}u+r_{1}\rho|u|u-\delta\rho\nabla\triangle^{3}\rho=\rho\nabla\Phi,\\
&\triangle\Phi=-4\pi G(\rho-1),
\end{array}\right.
\end{aligned}
\end{equation}
with suitable initial data, for any $T>0$. In particular, the weak solutions $(\rho,u,\Phi)$ satisfy the energy inequality (\ref{b35}) and (\ref{b30}).
\end{Proposition}
\section{B-D entropy and passing to the limits as $\varepsilon,\mu\rightarrow 0$}
In this section, we deduce the B-D entropy estimate for the approximate system in Proposition \ref{pro3} which was first introduced by Bresch and Desjardins in \cite{BDL2003}, this B-D entropy will give a higher regularity of the density and will help us to get the compactness of $\rho$. By (\ref{b26}),(\ref{b30}) and $u\in L^{2}(0,T;H^{2})$, we have
\begin{equation}\label{c1}
\begin{aligned}
\rho(x,t)\geq C(\delta,\eta)>0,\ \ \rho\in L^{\infty}(0,T;H^{3})\cap L^{2}(0,T;H^{4}),\ \ and \ \ \partial_{t}\rho\in L^{2}(0,T;L^{2}).
\end{aligned}
\end{equation}
\subsection{B-D entropy.}
Thanks to (\ref{c1}), it's easy to check $\frac{\nabla\rho}{\rho}\in L^{2}(0,T;H^{3}) $ and $\partial_{t}\frac{\nabla\rho}{\rho}\in L^{2}(0,T;L^{2})$, so we can use $\frac{\nabla\rho}{\rho}$ as a test function to test the momentum equation to derive the Bresch-Desjardns entropy. Thus, we have
\begin{Lemma}\label{l3}
\begin{equation}\label{c2}
\begin{aligned}
&\frac{d}{dt}\int(\frac{1}{2}\rho|\frac{\nabla\rho}{\rho}|^{2}+\rho u\frac{\nabla\rho}{\rho})dx+\delta\int|\triangle^{2}\rho|^{2}+\frac{2}{3}\eta\int
|\nabla\rho^{-3}|^{2}-\int \rho\partial_{i}u^{j}\partial_{j}u^{i}\\
&\ \ +\varepsilon\int\frac{1}{\rho}|\triangle \rho|^{2}\\
&=\int\rho\nabla\Phi\cdot\frac{\nabla\rho}{\rho}-\int\nabla P(\rho)\cdot\frac{\nabla\rho}{\rho}-r_{0}\int u\cdot\frac{\nabla\rho}{\rho}-r_{1}\int\rho |u|u\cdot\frac{\nabla\rho}{\rho}\\
&\ \ -\mu\int
\triangle u\cdot\triangle(\frac{\nabla\rho}{\rho})-\varepsilon\int {\rm div }(\rho u)\frac{\triangle\rho}{\rho}-\varepsilon\int\partial_{i}\rho\partial_{i}
u^{j}\frac{\partial_{j}\rho}{\rho}-\frac{\varepsilon}{2}\int \triangle\rho|\frac{\nabla\rho}{\rho}|^{2}\\
&=\sum_{i=1}^{8}I_{i},
\end{aligned}
\end{equation}
\end{Lemma}
firstly, we control the terms $I_{1}-I_{4}$:
\begin{equation}\label{c3}
\begin{aligned}
I_{1}=\int\rho\nabla\Phi\cdot\frac{\nabla\rho}{\rho}=-\int \triangle\Phi(\rho-1)=4\pi G\int (\rho-1)^{2},
\end{aligned}
\end{equation}
\begin{equation}\label{c4}
\begin{aligned}
I_{2}&=-\int \nabla P(\rho)\frac{\nabla\rho}{\rho}=-\int\frac{P^{'}}{\rho}|\nabla\rho|^{2}\\
&\leq -\int\frac{\frac{1}{a}\rho^{\gamma-1}-b}{\rho}|\nabla\rho|^{2}= -\frac{4}{a\gamma^{2}}\int|\nabla\rho^{\frac{\gamma}{2}}|^{2}+4b\int|\nabla\sqrt{\rho}|^{2},
\end{aligned}
\end{equation}
where we used the condition (\ref{a3}), besides this we can also have
\begin{equation}\label{c5}
\begin{aligned}
I_{3}&=-r_{0}\int u\cdot \frac{\nabla\rho}{\rho}=- r_{0}\int \rho ^{-1}{\rm div}(\rho u)=- r_{0}\int \rho ^{-1}(\varepsilon\triangle \rho-\rho_{t})\\
&=r_{0}\int \partial_{t}\log \rho-r_{0}\varepsilon\int\frac{1}{\rho^{2}}|\nabla\rho|^{2},
\end{aligned}
\end{equation}
\begin{equation}\label{c6}
\begin{aligned}
I_{4}=-r_{1}\int \rho|u|u\cdot\frac{\nabla\rho}{\rho}=r_{1}\int \rho(|u|{\rm div}u+u\frac{u}{|u|}\nabla u)\leq C\|\sqrt{\rho}u\|_{L^{2}}\|\sqrt{\rho}\nabla u\|_{L^{2}}.
\end{aligned}
\end{equation}
Substituting (\ref{c3})-(\ref{c6}) into (\ref{c2}) and integrating it with respect to the time $t$ over [0,T], we have
\begin{equation}\label{c7}
\begin{aligned}
&\frac{1}{2}\int \rho(u+\frac{\nabla\rho}{\rho})^{2}-r_{0}\log\rho\  dx+\varepsilon\int_{0}^{T}\int\frac{1}{\rho}|\triangle \rho|^{2}+r_{0}\varepsilon\int_{0}^{T}\int\rho^{-2}|\nabla\rho|^{2}\\
&+\frac{2}{3}\eta\int_{0}^{T}\int|\nabla
\rho^{-3}|^{2}+\int_{0}^{T}\int \rho|\nabla u|^{2}+\frac{4}{a\gamma^{2}}\int_{0}^{T}\int|\nabla\rho^{\frac{\gamma}{2}}|^{2}+\delta\int_{0}^{T}\int
|\triangle^{2} \rho|^{2}\\
&\leq \int\frac{1}{2}\rho |u|^{2}+\int_{0}^{T}\int\rho|\mathbb{D}u|^{2}+4b\int_{0}^{T}\int|\nabla\sqrt{\rho}|^{2}+4\pi G\int_{0}^{T}\int(\rho-1)^{2}\\
&+C\int_{0}^{T}\|\sqrt{\rho}u\|_{L^{2}}\|\sqrt{\rho}\nabla u\|_{L^{2}}+\int_{0}^{T}\sum_{i=5}^{8}I_{i}dt+\frac{1}{2}\int \rho_{0}(u_{0}+\frac{\nabla\rho_{0}}{\rho_{0}})^{2}-r_{0}\log\rho_{0}\  dx\\
&\leq C+C_{\varepsilon}+\frac{1}{2}\int_{0}^{T}\|\sqrt{\rho}\nabla u\|_{L^{2}}^{2}dt+4b\int_{0}^{T}\int|\nabla\sqrt{\rho}|^{2}+4\pi G\int_{0}^{T}\int(\rho-1)^{2}\\
&+\int_{0}^{T}\sum_{i=5}^{8}I_{i}dt,
\end{aligned}
\end{equation}
where we used the energy inequality (\ref{b35}). Then we need to control the rest terms on the right hand of the (\ref{c7}):
\begin{equation}\label{c8}
\begin{aligned}
4\pi G&\int_{0}^{T}\int(\rho-1)^{2}\leq C+C\int_{0}^{T}(\|\rho\|_{L^{1}}^{\frac{5\gamma-6}{2(5\gamma-3)}}\|\rho\|_{L^{\frac{5\gamma}{3}}}
^{\frac{5\gamma}{2(5\gamma-3)}})^{2}dt\\
&\leq C+C\int_{0}^{T}\|\rho\|_{L^{1}}^{\frac{5\gamma-6}{(5\gamma-3)}}\big(\|\rho^{\gamma}\|_{L^{1}}^{\frac
{2}{5}}\|\rho^{\gamma}\|_{L^{3}}^{\frac{3}{5}}\big)^{\frac{5}{5\gamma-3}}dt\\
&\leq C+C\int_{0}^{T}(C+C_{\varepsilon})^{\frac{2}{5\gamma-3}}\|\rho^{\gamma}\|_{L^{3}}
^{\frac{3}{5\gamma-3}}dt\\
&\leq \frac{1}{2}\frac{4}{a\gamma^{2}}\int_{0}^{T}\|\nabla\rho^{\frac{\gamma}{2}}\|_{L^{2}}^{2}dt+(C
+C_{\varepsilon})^{3},
\end{aligned}
\end{equation}
where we need to require $\frac{6}{5\gamma-3}<2$, which implies $\gamma>\frac{6}{5}$.
\begin{equation}\label{c9}
\begin{aligned}
&\int_{0}^{T}I_{5}dt=-\mu\int_{0}^{T}\int\triangle u\cdot\triangle(\frac{\nabla\rho}{\rho})\\
&=-\mu\int_{0}^{T}\int\triangle u\big[\frac{1}{\rho}{\rm div}\nabla^{2}\rho-\frac{\nabla\rho}{\rho^{2}}\nabla^{2}\rho-\frac{1}{\rho^{2}}{\rm div}(\nabla\rho\otimes\nabla\rho)+2\nabla\rho\otimes\nabla\rho\frac{\nabla\rho}{\rho^{3}}\big]\\
&\leq C\mu\int_{0}^{T}\int|\triangle u|\big[\frac{1}{\rho}|\nabla^{3}\rho|+\frac{1}{\rho^{2}}|\nabla\rho||\nabla^{2}\rho|+\frac{1}{\rho^{3}}
|\nabla\rho|^{3}\big]\\
&\leq C\sqrt{\mu}\|\sqrt{\mu}\triangle u\|_{L^{2}(L^{2})}\big[\|\frac{1}{\rho}\|_{L^{\infty}(L^{\infty})}\|\nabla^{3}\rho\|_{L^{2}(L^{2})}
+\|\frac{1}{\rho}\|_{L^{\infty}(L^{\infty})}^{2}\|\nabla\rho\|_{L^{\infty}(L^{\infty})}\|\nabla
^{2}\rho\|_{L^{2}(L^{2})}\\
&+\|\frac{1}{\rho}\|_{L^{\infty}(L^{\infty})}^{3}\|\nabla\rho\|_{L^{6}
(L^{6})}^{3}\big]\\
&\leq C\sqrt{\mu}\|\sqrt{\mu}\triangle u\|_{L^{2}(L^{2})}\big[\|\frac{1}{\rho}\|_{L^{\infty}(L^{\infty})}^{3}\|\nabla^{3}\rho\|_{L^{2}(L^{2})}
^{3}+\|\frac{1}{\rho}\|_{L^{\infty}(L^{\infty})}\big]\\
&\leq C(\delta,\eta)(C+C_{\varepsilon})^{s}\sqrt{\mu},
\end{aligned}
\end{equation}
for some large fixed constant $s>0$.
\begin{equation}\label{c10}
\begin{aligned}
&\int_{0}^{T}I_{6}dt=-\varepsilon\int_{0}^{T}\int {\rm div }(\rho u)\frac{\triangle \rho}{\rho}=-\varepsilon\int_{0}^{T}\int{\rm div }u\triangle \rho-\varepsilon\int_{0}^{T}\int\frac{u\cdot\nabla\rho}{\rho}\triangle\rho\\
&\leq C\varepsilon\|\nabla u\|_{L^{2}(L^{6})}\|\triangle \rho\|_{L^{2}(L^{\frac{6}{5}})}+C\varepsilon\|\rho^{\frac{1}{3}}u\|_{L^{3}(L^{3})}\|\rho^{-\frac
{4}{3}}\|_{L^{\infty}(L^{\infty})}\|\nabla\rho\|_{L^{6}(L^{6})}\|\triangle\rho\|_{L^{2}(L^{2})}\\
&\leq C\sqrt{\varepsilon}(\|\sqrt{\varepsilon}\nabla^{2}u\|_{L^{2}(L^{2})}+\|u\|_{L^{2}(L^{2})})
(\|\rho\|_{L^{\infty}(L^{1})}+\|\nabla^{3}\rho\|_{L^{\infty}(L^{2})})\\
&\ \ \ \ \ +C\varepsilon\|\rho^{\frac{1}{3}}u\|_{L^{3}(L^{3})}\|\rho^{-\frac
{4}{3}}\|_{L^{\infty}(L^{\infty})}(\|\rho\|_{L^{\infty}(L^{1})}+\|\nabla^{3}\rho\|_{L^{\infty}(L^{2})})^{2}\\
&\leq C(\delta,\eta,r_{0})(C+C_{\varepsilon})^{s}(\sqrt{\varepsilon}+\varepsilon),
\end{aligned}
\end{equation}
\begin{equation}\label{c11}
\begin{aligned}
&\int_{0}^{T}I_{7}dt=-\varepsilon\int\partial_{i}\rho\partial_{i}u^{j}\frac{\partial_{j}\rho}{\rho}\leq C\sqrt{\varepsilon}\|\rho^{-1}\|_{L^{\infty}(L^{\infty})}\|\nabla\rho\|_{L^{4}(L^{4})}^{2}\|
\nabla u\|_{L^{2}(L^{2})}\\
&\leq C(\delta,\eta)\sqrt{\varepsilon}\|\rho^{-1}\|_{L^{\infty}(L^{\infty})}\|\nabla^{3}\rho\|_{L^{\infty}
(L^{2})}^{2}\|\sqrt{\varepsilon}\nabla^{2} u\|_{L^{2}(L^{2})}\\
&\leq C(\delta,\eta,r_{0})(C+C_{\varepsilon})^{s}\sqrt{\varepsilon},
\end{aligned}
\end{equation}
and
\begin{equation}\label{c12}
\begin{aligned}
&\int_{0}^{T}I_{8}dt=-\frac{\varepsilon}{2}\int \triangle\rho|\frac{\nabla\rho}{\rho}|^{2}\leq C\varepsilon\int_{0}^{T}\|\rho^{-1}\|_{L^{\infty}}^{2}\|\triangle\rho\|_{L^{2}}\|\nabla\rho\|_{L^{4}}
^{2}dt\\
&\leq C\varepsilon\|\rho^{-1}\|_{L^{\infty}(L^{\infty})}^{2}(\|\nabla^{3}\rho\|_{L^{\infty}(L^{2})}+
\|\rho\|_{L^{\infty}(L^{1})})^{3}\\
&\leq C(\delta,\eta,)(C+C_{\varepsilon})^{s}\varepsilon.
\end{aligned}
\end{equation}
Then substituting (\ref{c8})-(\ref{c12}) into (\ref{c7}), we have
\begin{equation}\label{c13}
\begin{aligned}
&\frac{1}{2}\int \rho(u+\frac{\nabla\rho}{\rho})^{2}-r_{0}\log\rho dx+\varepsilon\int_{0}^{T}\int\frac{1}{\rho}|\triangle \rho|^{2}+r_{0}\varepsilon\int_{0}^{T}\int\rho^{-2}|\nabla\rho|^{2}\\
&+\frac{2}{3}\eta\int_{0}^{T}\int|\nabla
\rho^{-3}|^{2}+\frac{1}{2}\int_{0}^{T}\int \rho|\nabla u|^{2}+\frac{2}{a\gamma^{2}}\int_{0}^{T}\int|\nabla\rho^{\frac{\gamma}{2}}|^{2}+\delta\int_{0}^{T}\int
|\triangle^{2} \rho|^{2}\\
&\leq 4b\int_{0}^{T}\int|\nabla\sqrt{\rho}|^{2}+C(\delta,\eta,r_{0})(C+C_{\varepsilon})^{s}(\varepsilon+
\sqrt{\varepsilon}+\sqrt{\mu})+(C+C_{\varepsilon})^{3},
\end{aligned}
\end{equation}
by using the Gronwall inequality, yields
\begin{equation}\label{c14}
\begin{aligned}
&\int|\nabla\sqrt{\rho}|^{2}dx\leq [C(\delta,\eta,r_{0})(C+C_{\varepsilon})^{s}(\varepsilon+
\sqrt{\varepsilon}+\sqrt{\mu})+(C+C_{\varepsilon})^{3}](1+4bT\exp^{4bT}),
\end{aligned}
\end{equation}
so with this inequality and (\ref{c13}), we have the following B-D entropy estimats
\begin{equation}\label{c15}
\begin{aligned}
&\frac{1}{2}\int \rho(u+\frac{\nabla\rho}{\rho})^{2}-r_{0}\log\rho dx+\varepsilon\int_{0}^{T}\int|\triangle \rho|^{2}+r_{0}\varepsilon\int_{0}^{T}\int\rho^{-2}|\nabla\rho|^{2}\\
&+\frac{2}{3}\eta\int_{0}^{T}\int|\nabla
\rho^{-3}|^{2}+\int_{0}^{T}\int \rho|\nabla u|^{2}+\frac{4}{a\gamma^{2}}\int_{0}^{T}\int|\nabla\rho^{\frac{\gamma}{2}}|^{2}+\delta\int_{0}^{T}\int
|\triangle^{2} \rho|^{2}\\
&\leq[C(\delta,\eta,r_{0})(C+C_{\varepsilon})^{s}(\varepsilon+
\sqrt{\varepsilon}+\sqrt{\mu})+(C+C_{\varepsilon})^{3}][(1+4bT\exp^{4bT})+1],
\end{aligned}
\end{equation}
where $s>0$ is a suitable large fixed constant, $C$ is a generic positive constant depending on the initial data and other constants but independent of $\varepsilon,\delta,\eta,r_{0}$, and $C_{\varepsilon}=C(\frac{\varepsilon}{\delta}+\frac{\varepsilon}{\eta})[1+(\frac{\varepsilon}{\delta}
+\frac{\varepsilon}{\eta})Te^{C(\frac{\varepsilon}{\delta}
+\frac{\varepsilon}{\eta})T}]T.$

Thus, at this approximation level, with $\delta,\eta,r_{0}$ being fixed and $\varepsilon\ll 1,\mu\ll 1$,  from (\ref{b35}) and (\ref{c15}) we get the following energy inequality and the B-D entropy
\begin{equation}\label{c16}
\begin{aligned}
&\sup_{t\in (0,T)}\int\frac{1}{2}\rho_{\mu,\varepsilon} u_{\mu,\varepsilon}^{2}+\frac{1}{a\gamma(\gamma-1)}\rho_{\mu,\varepsilon}^{\gamma}+\frac{\eta}{7}
\rho_{\mu,\varepsilon}^{-6}+\frac{\delta}{2}|\nabla
\triangle\rho_{\mu,\varepsilon}|^{2}dx+\frac{4\varepsilon}{a\gamma^{2}}\int_{0}^{T}\int
|\nabla\rho_{\mu,\varepsilon}^{\frac{\gamma}{2}}
|^{2}\\
&+\frac{2}{3}\eta\varepsilon\int_{0}^{T}\int|\nabla \rho_{\mu,\varepsilon}^{-3}|^{2}+\int_{0}^{T}\int\rho_{\mu,\varepsilon} |\mathbb{D}u_{\mu,\varepsilon}|^{2}+r_{0}\int_{0}^{T}\int|u_{\mu,\varepsilon}|^{2}
+r_{1}\int_{0}^{T}\int\rho_{\mu,\varepsilon}|u_{\mu,\varepsilon}|^{3}\\
&+\mu\int_{0}^{T}\int|\triangle u_{\mu,\varepsilon}|^{2}+\delta\varepsilon\int_{0}^{T}\int|\triangle^{2}\rho_{\mu,\varepsilon}|^{2}\leq C(\delta,\eta,T),
\end{aligned}
\end{equation}
and
\begin{equation}\label{c17}
\begin{aligned}
&\frac{1}{2}\int \rho_{\mu,\varepsilon}(u_{\mu,\varepsilon}+\frac{\nabla\rho_{\mu,\varepsilon}}{\rho_{\mu,\varepsilon}})^{2}
-r_{0}\log\rho_{\mu,\varepsilon} dx+\varepsilon\int_{0}^{T}\int\frac{1}{\rho_{\mu,\varepsilon}}|\triangle \rho_{\mu,\varepsilon}|^{2}+r_{0}\varepsilon\int_{0}^{T}\int\rho_{\mu,\varepsilon}^{-2}
|\nabla\rho_{\mu,\varepsilon}|^{2}\\
&+\frac{2}{3}\eta\int_{0}^{T}\int|\nabla
\rho_{\mu,\varepsilon}^{-3}|^{2}+\frac{1}{2}\int_{0}^{T}\int \rho_{\mu,\varepsilon}|\nabla u|^{2}+\frac{2}{a\gamma^{2}}\int_{0}^{T}\int|\nabla\rho_{\mu,\varepsilon}^{\frac{\gamma}{2}}|^{2}
+\delta\int_{0}^{T}\int
|\triangle^{2} \rho_{\mu,\varepsilon}|^{2}\\
&\leq C(\delta,\eta,T),
\end{aligned}
\end{equation}
where $C(\delta,\eta,T)$ denotes that $C$ particularly depends on $\delta,\eta$ and time $T$.
\subsection{Passing to the limits as $\mu,\varepsilon\rightarrow0$.} We use $(\rho_{\mu,\varepsilon},u_{\mu,\varepsilon},\Phi(\rho_{\mu,\varepsilon}))$ to denote the solutions at this level of approximation. From (\ref{c16}), (\ref{c17}), it's easy to show that $(\rho_{\mu,\varepsilon},u_{\mu,\varepsilon},\Phi(\rho_{\mu,\varepsilon}))$ has the following uniform regularities
\begin{equation}\label{c18}
\begin{aligned}
&\sqrt{\rho_{\mu,\varepsilon}}u_{\mu,\varepsilon}\in L^{\infty}(0,T;L^{2}), \sqrt{\rho_{\mu,\varepsilon}}\mathbb{D}u_{\mu,\varepsilon}\in L^{2}(0,T;L^{2}), \\
&\rho_{\mu,\varepsilon}^{\gamma}\in L^{\infty}(0,T;L^{1}),\ \sqrt{\varepsilon}\nabla\rho_{\mu,\varepsilon}^{\frac{\gamma}{2}}\in L^{2}(0,T;L^{2}),\ \eta\rho_{\mu,\varepsilon}^{-6}\in L^{\infty}(0,T;L^{1}),\\
&\sqrt{\varepsilon}\nabla\rho_{\mu,\varepsilon}^{-3}\in L^{2}(0,T;L^{2}),\ \rho_{\mu,\varepsilon}\in L^{\infty}(0,T;H^{3}),\ \sqrt{\varepsilon}\rho_{\mu,\varepsilon}\in L^{2}(0,T;H^{4}),\\
&\ u_{\mu,\varepsilon}\in L^{2}(0,T;L^{2}),\ \rho_{\mu,\varepsilon}^{\frac{1}{3}}u_{\mu,\varepsilon}\in L^{3}(0,T;L^{3}),\sqrt{\mu}\triangle u_{\mu,\varepsilon}\in L^{2}(0,T;L^{2}),
\end{aligned}
\end{equation}
by the Bresch-Dejardins entropy, we also have the following additional regularities
\begin{equation}\label{c19}
\begin{aligned}
&\nabla\sqrt{\rho_{\mu,\varepsilon}}\in L^{\infty}(0,T;L^{2}),\ \sqrt{\delta}\rho_{\mu,\varepsilon}\in L^{2}(0,T;H^{4}),\ \nabla\rho_{\mu,\varepsilon}^{\frac{\gamma}{2}}\in L^{2}(0,T;L^{2})\\
&\sqrt{\eta}\nabla\rho_{\mu,\varepsilon}^{-3}\in L^{2}(0,T;L^{2}),\ \sqrt{\rho_{\mu,\varepsilon}}\nabla u_{\mu,\varepsilon}\in L^{2}(0,T;L^{2}).
\end{aligned}
\end{equation}
Based on the above regularities, we have the following uniform compactness results:
\begin{Lemma}\label{l4}
Let $(\rho_{\mu,\varepsilon},u_{\mu,\varepsilon},\Phi(\rho_{\mu,\varepsilon}))$ be weak solutions to (\ref{b36}), in combination with (\ref{c18}) and (\ref{c19}), we have
\begin{equation}\label{c20}
\begin{aligned}
&\rho_{\mu,\varepsilon}\in L^{2}(0,T;H^{4}),\ \ \partial_{t}\rho_{\mu,\varepsilon}\in L^{2}(0,T;H^{-1}),\\
&\rho_{\mu,\varepsilon}u_{\mu,\varepsilon}\in L^{2}(0,T;W^{1,\frac{3}{2}}),\ \ \partial_{t}(\rho_{\mu,\varepsilon}u_{\mu,\varepsilon})\in\ L^{2}(0,T;H^{-3}),\\
&\rho_{\mu,\varepsilon}^{\gamma}\in L^{\frac{5}{3}}(0,T;L^{\frac{5}{3}}),\ \rho_{\mu,\varepsilon}^{-6}\in L^{\frac{5}{3}}(0,T;L^{\frac{5}{3}}),\\
&P(\rho_{\mu,\varepsilon})\in L^{\frac{5}{3}}(0,T;L^{\frac{5}{3}}),\ \Phi(\rho_{\mu,\varepsilon})\in L^{\infty}(0,T;H^{2}),
\end{aligned}
\end{equation}
and using Aubin-Lions Lemma, we have the following compactness results
\begin{equation}\label{c21}
\begin{aligned}
&\rho_{\mu,\varepsilon}\rightarrow \rho,\ a.e.\ and\ strongly\ in\  C([0,T];H^{3}),\ \rho_{\mu,\varepsilon}\rightharpoonup \rho,\ in\ L^{2}(0,T;H^{4}),\ \\
&P(\rho_{\mu,\varepsilon})\rightarrow P(\rho),\ a.e.\ and\ strongly\ in\  L^{1}(0,T;L^{1}),\\
&u_{\mu,\varepsilon}\rightharpoonup u,\ in\ L^{2}(0,T;L^{2}),\\
&\rho_{\mu,\varepsilon}u_{\mu,\varepsilon}\rightarrow \rho u\ strongly\ in\ L^{2}(0,T;L^{p}),\ for\ \forall 1\leq p<3,\\
&\rho_{\mu,\varepsilon}^{-6}\rightarrow \rho^{-6},\ a.e.\ and\ strongly\ in\  L^{1}(0,T;L^{1}),\\
&\nabla\Phi(\rho_{\mu,\varepsilon}) \rightarrow \nabla \Phi(\rho) \ strongly\ in\ L^{2}(0,T;L^{2}),\\
&-\rho_{\mu,\varepsilon}\mathbb{D}u_{\mu,\varepsilon}\rightarrow -\rho\mathbb{D}u,\ in\ the\ sense\ of\ distribution\ on\ (0,T)\times\mathbb{T}^{3},\\
&-\delta\rho_{\mu,\varepsilon}\nabla\triangle^{3}\rho_{\mu,\varepsilon}\rightarrow-\delta\rho
\nabla\triangle^{3}\rho,\ in\ the\ sense\ of\ distribution\ on\ (0,T)\times\mathbb{T}^{3},\\
\end{aligned}
\end{equation}\end{Lemma}
\begin{proof}
The proof is similar to the compactness analysis in section 2, we repeat the compactness arguments here again and for the simplicity, we omit the details here.
\end{proof}

With the above compactness results in hand, then we pass to the limits as $\mu=\varepsilon\rightarrow0$, here we only focus on the terms involving with $\varepsilon$ and $\mu$.
Firstly, because $\rho_{\mu,\varepsilon}$ is bounded in $L^{\infty}(H^{3})\cap L^{2}(H^{4})$ uniformly on $\varepsilon$, we have
\begin{equation}\label{c22}
\begin{aligned}
\varepsilon\int_{0}^{T}\int \triangle \rho_{\mu,\varepsilon}\varphi\leq \varepsilon\|\triangle\rho_{\mu,\varepsilon}\|_{L^{2}(L^{2})}\|\varphi\|_{L^{2}(L^{2})}
\rightarrow0,\ \ as\ \ \varepsilon\rightarrow0,
\end{aligned}
\end{equation}
for any test function $\varphi\in C_{per}^{\infty}([0,T];\mathbb{T}^{3})$. So passing to the limits in $(\ref{b36})_{1}$ and using the Lemma \ref{l4}, we have
\begin{equation}\label{c23}
\rho_{t}+{\rm div}(\rho u)=0,\ \ a.e.\ \ in\  (0,T)\times \mathbb{T}^{3}.
\end{equation}
Similarly,
\begin{equation}\label{c24}
\begin{aligned}
\varepsilon\int_{0}^{T}\int& \nabla\rho_{\mu,\varepsilon}\cdot\nabla u_{\mu,\varepsilon}\varphi\\
&\leq \sqrt{\varepsilon}\|\nabla\rho_{\mu,\varepsilon}\|_{L^{2}(L^{2})}\|\sqrt{\varepsilon}\nabla u_{\mu,\varepsilon}\|
_{L^{2}(L^{2})}\|\varphi\|_{L^{\infty}(L^{\infty})}\rightarrow0,\ \ as\ \ \varepsilon\rightarrow0,
\end{aligned}
\end{equation}
and
\begin{equation}\label{c25}
\begin{aligned}
\mu\int_{0}^{T}\int\triangle^{2}u_{\mu,\varepsilon}\varphi
&=\mu\int_{0}^{T}\int\triangle
u_{\mu,\varepsilon}\triangle\varphi\\
&\leq \sqrt{\mu}\|\sqrt{\mu}\triangle u_{\mu,\varepsilon}\|_{L^{2}(L^{2})}\|\triangle \varphi\|_{L^{2}(L^{2})}\rightarrow0,\ as\ \mu\rightarrow0.
\end{aligned}
\end{equation}

So pass to the limits as $\mu=\varepsilon\rightarrow0$ in $(\ref{b36})$, we have
\begin{equation}\label{c26}
\begin{aligned}
&(\rho u)_{t}+{\rm div}(\rho u\otimes u)+\nabla P(\rho)-\eta\nabla\rho^{-6}-{\rm div}(\rho \mathbb{D}u)+ r_{0}u+r_{1}\rho|u|u\\
& -\delta\rho\nabla\triangle^{3}\rho=\rho\nabla\Phi,\ holds\ in\ the \ sense\ of\ distribution\ on\ (0,T)\times\mathbb{T}^{3},
\end{aligned}
\end{equation}
and
\begin{equation}\label{c27}
\triangle\Phi=-4\pi G(\rho-1),\ \ holds\ a.e.\ on\ (0,T)\times\mathbb{T}^{3}.
\end{equation}

Furthermore, thanks to lower semi-continuity of the convex function and the strong convergence of $\rho_{\mu,\varepsilon},u_{\mu,\varepsilon},\Phi(\rho_{\mu,\varepsilon})$, we can pass to the limits in the energy inequality (\ref{b35}) and Bresch-Desjardins entropy (\ref{c15}) as $\mu=\varepsilon\rightarrow0$ with $\delta,\eta,r_{0}$ being fixed,
\begin{equation}\label{c28}
\begin{aligned}
\int\frac{1}{2}\rho &u^{2}+\frac{1}{a\gamma(\gamma-1)}\rho^{\gamma}+\frac{\eta}{7}\rho^{-6}+\frac{\delta}{2}|\nabla
\triangle\rho|^{2}dx+\int_{0}^{T}\int\rho |\mathbb{D}u|^{2}\\
&\ \ \ \ +r_{0}\int_{0}^{T}\int|u|^{2}+r_{1}\int_{0}^{T}\int\rho|u|^{3}
\leq C(T),
\end{aligned}
\end{equation}
and
\begin{equation}\label{c29}
\begin{aligned}
&\frac{1}{2}\int \rho(u+\frac{\nabla\rho}{\rho})^{2}-r_{0}\log\rho dx+\frac{2}{3}\eta\int_{0}^{T}\int|\nabla
\rho^{-3}|^{2}+\frac{1}{2}\int_{0}^{T}\int \rho|\nabla u|^{2}\\
&\ \ \ \ \ \ \ \ \ \ \ \ \ \ \ +\frac{2}{a\gamma^{2}}\int_{0}^{T}\int|\nabla\rho^{\frac{\gamma}{2}}|^{2}+\delta\int_{0}^{T}
\int|\triangle^{2} \rho|^{2}\leq C(T),
\end{aligned}
\end{equation}
where $C(T)$ is a generic positive constant independent of $\varepsilon,\eta,\delta,r_{0}$, and we used the fact that $C_{\varepsilon}=C(\frac{\varepsilon}{\delta}+\frac{\varepsilon}{\eta})[1+(\frac{\varepsilon}{\delta}
+\frac{\varepsilon}{\eta})Te^{C(\frac{\varepsilon}{\delta}
+\frac{\varepsilon}{\eta})T}]T\rightarrow0$, as $\varepsilon\rightarrow0$.

Thus, to conclude this part, we have the following proposition
\begin{Proposition}\label{pro4}
There exists the weak solutions to the system (\ref{c23}), (\ref{c26}) and (\ref{c27}) with suitable initial data, for any $T>0$. In particular, the weak solutions $(\rho,u,\Phi)$ satisfy the energy inequality (\ref{c28}) and the B-D entropy (\ref{c29}).

\end{Proposition}
\section{Passing to the limits as $\eta\rightarrow0$.}
In this section, we pass to the limits as $\eta\rightarrow0$ with $\delta,r_{0}$ being fixed. we denote that $(\rho_{\eta},u_{\eta},\Phi(\rho_{\eta}))$ are weak solutions at this level, from the proposition \ref{pro4}, we have the following regularities
\begin{equation}\label{d1}
\begin{aligned}
&\sqrt{\rho_{\eta}}u_{\eta}\in L^{\infty}(0,T;L^{2}),\ \sqrt{\rho_{\eta}}\mathbb{D}u_{\eta}\in L^{2}(0,T;L^{2}),\\
&\nabla\sqrt{\rho_{\eta}}\in L^{\infty}(0,T;L^{2}),\ \sqrt{\delta}\rho_{\eta}\in L^{\infty}(0,T;H^{3}),\ \sqrt{\delta}\rho_{\eta}\in L^{2}(0,T;H^{4}),\\
&u_{\eta}\in L^{2}(0,T;L^{2}),\ \rho_{\eta}^{\frac{1}{3}}u_{\eta}\in L^{3}(0,T;L^{3}),\ \sqrt{\rho_{\eta}}\nabla u_{\eta}\in L^{2}(0,T;L^{2}),\\
&\rho_{\eta}^{\gamma}\in L^{\infty}(0,T;L^{1}),\ \nabla\rho_{\eta}^{\frac{\gamma}{2}}\in L^{2}(0,T;L^{2}),\\
&\eta\rho_{\eta}^{-6}\in L^{\infty}(0,T;L^{1}),\ \sqrt{\eta}\nabla \rho_{\eta}^{-3}\in L^{2}(0,T;L^{2}).\\
\end{aligned}
\end{equation}

So it's easy to check that we have the same estimates as in Lemma \ref{l4} at the level with $\eta$, thus we deduce the same compactness for $(\rho_{\eta},u_{\eta},\Phi(\rho_{\eta}))$ as follows
\begin{equation}\label{d2}
\begin{aligned}
&\rho_{\eta}\rightarrow \rho,\ a.e.\ and\ strongly\ in\  C([0,T];H^{3}),\ \rho_{\eta}\rightharpoonup \rho,\ in\ L^{2}(0,T;H^{4}),\ \\
&P(\rho_{\eta})\rightarrow P(\rho),\ a.e.\ and\ strongly\ in\  L^{1}(0,T;L^{1}),\\ &u_{\eta}\rightharpoonup u,\ in\ L^{2}(0,T;L^{2}),\\
&\rho_{\eta}u_{\eta}\rightarrow \rho u\ \ a.e.\ and\ strongly\ in\ L^{2}(0,T;L^{2}),\\
&\nabla\Phi(\rho_{\eta}) \rightarrow \nabla \Phi(\rho) \ strongly\ in\ L^{2}(0,T;L^{2}),\\
&-\rho_{\eta}\mathbb{D}u_{\eta}\rightarrow -\rho\mathbb{D}u,\ in\ the\ sense\ of\ distribution\ on\ (0,T)\times\mathbb{T}^{3},\\
&-\delta\rho_{\eta}\nabla\triangle^{3}\rho_{\eta}\rightarrow-\delta\rho
\nabla\triangle^{3}\rho,\ in\ the\ sense\ of\ distribution\ on\ (0,T)\times\mathbb{T}^{3}.\\
\end{aligned}
\end{equation}
 So at this level of approximation, we only focus on the convergence of the term $\eta\nabla\rho_{\eta}^{-6}$. Here we state the following Lemma.
 \begin{Lemma}\label{l5}
 For $\rho_{\eta}$ defined as in Proposition \ref{pro4}, we have
 $$\eta\int_{0}^{T}\int \rho_{\eta}^{-6}dxdt\rightarrow0$$
 as $\eta\rightarrow0$.
 \end{Lemma}
 \begin{proof}
 The proof is inspired by Vasseur and Yu in \cite{vy2015}. From the B-D entropy (\ref{c29}), we have
 \begin{equation}\label{d3}
 \sup_{t\in[0,T]}\int (\log(\frac{1}{\rho_{\eta}}))_{+}dx\leq C(r_{0})<+\infty.
 \end{equation}
 Note that
 $$y\in \mathbb{R}^{+}\rightarrow\log(\frac{1}{y})_{+}$$
 is a convex continuous function. Moreover, in combination with the property of the convex function and Fatou's Lemma, yields
 \begin{equation}\label{d4}
 \begin{aligned}
 \int(\log(\frac{1}{\rho}))_{+}dx&\leq \int \lim_{\eta\rightarrow0} \inf(\log(\frac{1}{\rho_{\eta}}))_{+}dx\\
 &\leq \lim_{\eta\rightarrow0}\inf \int (\log(\frac{1}{\rho_{\eta}}))_{+}dx,
 \end{aligned}
 \end{equation}
 which implies $(\log(\frac{1}{\rho}))_{+}$ is bounded in $L^{\infty}(0,T;L^{1})$, so it allows us to deduce that
 \begin{equation}\label{d5}
 |\{x\ |\ \rho(t,x)=0\}|=0,\ \ for\ almost\ every\ t\in [0,T],
 \end{equation}
 where $|A|$ denotes the measure of set $A$.

 Thanks to the compactness of the density: $\rho_{\eta}\rightarrow\rho$ strongly in $C([0,T];H^{3})$, hence $\rho_{\eta}\rightarrow\rho$ a.e., then together with (\ref{d5}), we deduce
 \begin{equation}\label{d6}
 \eta\rho_{\eta}^{-6}\rightarrow0\ \ a.e.
 \end{equation}

 Moreover,using the interpolation inequality, yields
 $$\|\eta\rho_{\eta}^{-6}\|_{L^{\frac{5}{3}}(0,T;L^{\frac{5}{3}})}\leq \|\eta\rho_{\eta}^{-6}\|_{L^{\infty}(0,T;L^{1})}^{\frac{2}{5}}
 \|\eta\rho_{\eta}^{-6}\|_{L^{1}(0,T;L^{3})}^{\frac{3}{5}}\leq C,$$
 this together with (\ref{d6}) and using the Eogroffs theorem, yields
 $$\eta\rho_{\eta}^{-6}\rightarrow0,\ strongly\ in\ L^{1}(0,T;L^{1}).$$
 \end{proof}

 Thus, by using the compactness results (\ref{d2}), we can pass to the limit as $\eta\rightarrow0$ in (\ref{c23}),(\ref{c26}) and (\ref{c27}), yields
 \begin{equation}\label{d7}
 \begin{aligned}
 &\rho_{t}+{\rm div}(\rho u)=0,\ \ holds\ a.e.\ on\ (0,T)\times\mathbb{T}^{3},\\
&(\rho u)_{t}+{\rm div}(\rho u\otimes u)+\nabla P(\rho)-{\rm div}(\rho \mathbb{D}u)+ r_{0}u+r_{1}\rho|u|u\\
& -\delta\rho\nabla\triangle^{3}\rho=\rho\nabla\Phi,\ holds\ in\ the \ sense\ of\ distribution\ on\ (0,T)\times\mathbb{T}^{3},\\
&\triangle\Phi=-4\pi G(\rho-1),\ \ holds\ a.e.\ on\ (0,T)\times\mathbb{T}^{3}.
  \end{aligned}
 \end{equation}
 Similarly, due to the lower semi-continuity of convex functions, we can obtain the energy inequality and B-D entropy by passing to the limits in (\ref{c28}) and (\ref{c29}) as $\eta\rightarrow0$, we have
 \begin{equation}\label{d8}
\begin{aligned}
\int\frac{1}{2}\rho &u^{2}+\frac{1}{a\gamma(\gamma-1)}\rho^{\gamma}+\frac{\delta}{2}|\nabla
\triangle\rho|^{2}dx+\int_{0}^{T}\int\rho |\mathbb{D}u|^{2}\\
&\ \ \ \ +r_{0}\int_{0}^{T}\int|u|^{2}+r_{1}\int_{0}^{T}\int\rho|u|^{3}
\leq C(T),
\end{aligned}
\end{equation}
and
\begin{equation}\label{d9}
\begin{aligned}
&\frac{1}{2}\int \rho(u+\frac{\nabla\rho}{\rho})^{2}-r_{0}\log\rho dx+\frac{1}{2}\int_{0}^{T}\int \rho|\nabla u|^{2}\\
&\ \ \ \ \ \ \ \ \ \ \ \ \ \ \ +\frac{2}{a\gamma^{2}}\int_{0}^{T}\int|\nabla\rho^{\frac{\gamma}{2}}|^{2}+\delta\int_{0}^{T}
\int|\triangle^{2} \rho|^{2}\leq C(T),
\end{aligned}
\end{equation}
Thus we have the following Proposition on the existence of the weak solutions at this level of approximation.
\begin{Proposition}\label{pro5}
There exist weak solutions to the system (\ref{d7}) with suitable initial data, for any $T>0$. In particular, the weak solutions $(\rho,u,\Phi(\rho))$ satisfy the energy inequality (\ref{d8}) and the B-D entropy (\ref{d9}).
\end{Proposition}
\section{Passing to the limits as $\delta,r_{0}\rightarrow0$.}
At this level, the weak solutions satisfy the energy inequality (\ref{d8}) and the B-D entropy (\ref{d9}), thus we have the following regularities:
\begin{equation}\label{e1}
\begin{aligned}
&\sqrt{\rho_{\delta,r_{0}}}u_{\rho_{\delta,r_{0}}}\in L^{\infty}(0,T;L^{2}),\ \sqrt{\rho_{\rho_{\delta,r_{0}}}}\mathbb{D}u_{\rho_{\delta,r_{0}}}\in L^{2}(0,T;L^{2}),\ \\
&\nabla \sqrt{\rho_{\rho_{\delta,r_{0}}}}\in L^{\infty}(0,T;L^{2}),\ \sqrt{\delta}\rho_{\rho_{\delta,r_{0}}}\in L^{\infty}(0,T;H^{3})\cap L^{2}(0,T;H^{4}),\\
&\rho_{\rho_{\delta,r_{0}}}^{\gamma}\in L^{\infty}(0,T;L^{1}),\ \nabla\rho_{\rho_{\delta,r_{0}}}^{\frac{\gamma}{2}}\in L^{2}(0,T;L^{2}),\ \sqrt{\rho_{\rho_{\delta,r_{0}}}}\nabla u_{\rho_{\delta,r_{0}}}\in L^{2}(0,T;L^{2}),\\
&\sqrt{r_{0}}u_{\rho_{\delta,r_{0}}}\in L^{2}(0,T;L^{2}),\ \rho_{\rho_{\delta,r_{0}}}^{\frac{1}{3}}u_{\rho_{\delta,r_{0}}}\in L^{3}(0,T;L^{3}),
\end{aligned}
\end{equation}

Next, we will proceed the compactness arguments in several steps
\subsection{ Step 1: Convergence of $\sqrt{\rho_{\rho_{\delta,r_{0}}}}$.}
\begin{Lemma}\label{l6}
Let $(\rho_{\rho_{\delta,r_{0}}},u_{\rho_{\delta,r_{0}}},\Phi(\rho_{\rho_{\delta,r_{0}}}))$ satisfy the Proposition \ref{pro5}, we have
$$\sqrt{\rho_{\rho_{\delta,r_{0}}}} is\ bounded\  in\  L^{\infty}(0,T;H^{1}),$$
$$\partial_{t}\sqrt{\rho_{\rho_{\delta,r_{0}}}} is\ bounded\  in\  L^{2}(0,T;H^{-1}).$$
As a consequence, up to a subsequence, $\sqrt{\rho_{\rho_{\delta,r_{0}}}}$ convergences almost everywhere and strongly in $L^{2}(0,T;L^{2})$, which means
$$\sqrt{\rho_{\rho_{\delta,r_{0}}}}\rightarrow \sqrt{\rho},\ \ a.e.\ and\ strongly\ in\ L^{2}(0,T;L^{2}).$$
Moreover, we have
$$\rho_{\rho_{\delta,r_{0}}}\rightarrow \rho\ \ a.e.\ and\ strongly\ in\ C([0,T];L^{p}),\ \ for \ any\ p\in [1,3).$$

\end{Lemma}
\begin{proof}
In combination with the conservation of the mass $\|\rho_{\delta,r_{0}}(t)\|_{L^{1}}=\|\rho_{\rho_{\delta,r_{0}}}(0)\|_{L^{1}}$ and estimate in (\ref{e1}) gives $\sqrt{\rho_{\rho_{\delta,r_{0}}}} \in L^{\infty}(0,T;H^{1})$. Next, we notice that
\begin{equation}\label{e2}
\begin{aligned}
\partial_{t}\sqrt{\rho_{\delta,r_{0}}}&=-\frac{1}{2}\sqrt{\rho_{\delta,r_{0}}}{\rm div} u_{\delta,r_{0}}-u_{\delta,r_{0}}\nabla\sqrt{\rho_{\delta,r_{0}}}\\
&=\frac{1}{2}\sqrt{\rho_{\delta,r_{0}}}{\rm div} u_{\delta,r_{0}}-{\rm div}(u_{\delta,r_{0}}\sqrt{\rho_{\delta,r_{0}}}),
\end{aligned}
\end{equation}
which yields $\partial_{t}\sqrt{\rho_{\delta,r_{0}}}\in L^{2}(0,T;H^{-1})$, thanks to the Aubin-Lions Lemma, we have
$$\sqrt{\rho_{\delta,r_{0}}}\rightarrow \sqrt{\rho},\ \  strongly\ in\ L^{2}(0,T;L^{2}),$$
and hence yields $\sqrt{\rho_{\delta,r_{0}}}\rightarrow \sqrt{\rho}\ a.e.$.
In another hand, since $\nabla{\sqrt{\rho_{\delta,r_{0}}}}$ is bounded in $L^{\infty}(0,T;L^{2})$, by using the Sobolev embedding theorem, we have
$$\|\sqrt{\rho_{\delta,r_{0}}}\|_{L^{\infty}(L^{6})}\leq C\|\nabla\sqrt{\rho_{\delta,r_{0}}}\|_{L^{\infty}(L^{2})}<+\infty,$$
so
\begin{equation}\label{e3}
\begin{aligned}
\rho_{\delta,r_{0}}u_{\delta,r_{0}}=\sqrt{\rho_{\delta,r_{0}}}\sqrt{\rho_{\delta,r_{0}}}
u_{\delta,r_{0}}\in L^{\infty}(0,T;L^{\frac{3}{2}}),
\end{aligned}
\end{equation}
which yields that ${\rm div}(\rho_{\delta,r_{0}}u_{\delta,r_{0}})\in L^{\infty}(0,T;W^{-1,\frac{3}{2}})$, so the continuity equation yields $\partial_{t}\rho_{\delta,r_{0}}\in L^{\infty}(0,T;W^{-1,\frac{3}{2}})$.
Furthermore, because
\begin{equation}\label{e4}
\begin{aligned}
\nabla{\rho_{\delta,r_{0}}}=2\sqrt{\rho_{\delta,r_{0}}}\nabla{\sqrt{\rho_{\delta,r_{0}}}}\in L^{\infty}(0,T;L^{\frac{3}{2}}),
\end{aligned}
\end{equation}
 so we have $\rho_{\delta,r_{0}}$ is bounded in $L^{\infty}(0,T;W^{1,\frac{3}{2}})$.

  Then together with $\partial_{t}\rho_{\delta,r_{0}}\in L^{\infty}(0,T;W^{-1,\frac{3}{2}})$, thanks to the Aubin-Lions Lemma gives
\begin{equation}\label{e6}
\rho_{\delta,r_{0}}\rightarrow \rho,\ strongly\ in\ C([0,T];L^{p}),\ for\ any\ p\in[1,3),
\end{equation}
and hence, we have
$$\rho_{\delta,r_{0}}\rightarrow \rho\ a.e..$$
Thus the proof of this Lemma is completed.
\end{proof}
\subsection{ Step 2: Convergence of the pressure.}
\begin{Lemma}\label{l7}
The pressure $P(\rho_{\delta,r_{0}})$ satisfies the following regularity:
$$P(\rho_{\delta,r_{0}})\in L^{\frac{5}{3}}(0,T;L^{\frac{5}{3}}),$$
and up to subsequence, we have
$$P(\rho_{\delta,r_{0}})\rightarrow P(\rho)\ \ a.e.,$$
and
$$P(\rho_{\delta,r_{0}})\rightarrow P(\rho)\ \ strongly\ in\ L^{1}(0,T;L^{1}).$$

\end{Lemma}
\begin{proof}
The proof is as the same as it in section 2, so we omit the details here.
\end{proof}
\subsection{Step 3: Convergence of the momentum.}
\begin{Lemma}\label{l8}
Up to a subsequence, the momentum $m_{\delta,r_{0}}=\rho_{\delta,r_{0}}u_{\delta,r_{0}}$
 converges strongly in $L^{2}(0,T;L^{q})$ to some $m(x,t)$ for all $q\in [1,\frac{3}{2})$. In particular
 $$\rho_{\delta,r_{0}}u_{\delta,r_{0}}\rightarrow m\ a.e. \ for\ (x,t)\in \mathbb{T}^{3}\times(0,T).$$
 Note that we can define $u(x,t)=m(x,t)/\rho(x,t)$ outside the vacuum set $\{x|\rho(x,t)=0\}$.
\end{Lemma}
\begin{proof}
Since
\begin{equation}\label{e7}
\begin{aligned}
\nabla&(\rho_{\delta,r_{0}}u_{\delta,r_{0}})=\nabla{\rho_{\delta,r_{0}}}u_{\delta,r_{0}}+\rho
_{\delta,r_{0}}\nabla{u_{\delta,r_{0}}}\\
&=2\nabla{\sqrt
{\rho_{\delta,r_{0}}}}
\sqrt{\rho_{\delta,r_{0}}}u_{\delta,r_{0}}+\sqrt{\rho_{\delta,r_{0}}}\sqrt{\rho_{\delta,r_{0}}}
\nabla{u_{\delta,r_{0}}}\in L^{2}(0,T;L^{1}),
\end{aligned}
\end{equation}
together with (\ref{e3}), yields
$$\rho_{\delta,r_{0}}u_{\delta,r_{0}}\in L^{2}(0,T;W^{1,1}).$$
In order to apply the Aubin-Lions Lemma, we also need to show
$$\partial_{t}(\rho_{\delta,r_{0}}u_{\delta,r_{0}})\ is\  bounded \ in\  L^{2}(0,T;H^{-s}),\ \ for\ some \ constant\ s>0,$$
actually, use the momentum equation $(\ref{d7})_{2}$, it's easy to check that
$$\partial_{t}(\rho_{\delta,r_{0}}u_{\delta,r_{0}})\ is\  bounded \ in\  L^{2}(0,T;H^{-3}).$$
Hence, using the Aubin-Lions Lemma, the Lemma \ref{l8} is proved.
\end{proof}
\subsection{Step 4: Convergence of $\sqrt{\rho_{\delta,r_{0}}}u_{\delta,r_{0}}$.}
\begin{Lemma}\label{l9}
We have
$$\sqrt{\rho_{\delta,r_{0}}}u_{\delta,r_{0}}\rightarrow m/\sqrt{\rho},\ \ strongly\ in\ L^{2}(0,T;L^{2}).$$
In particular, we have $m(x,t)=0$ a.e. on $\{x\ |\ \rho(x,t)=0\}$ and there exists a function $u(x,t)$ such that $m(x,t)=\rho(x,t)u(x,t)$ and
$$\sqrt{\rho_{\delta,r_{0}}}u_{\delta,r_{0}}\rightarrow \sqrt{\rho}u,\ \ strongly\ in\ L^{2}(0,T;L^{2}).$$

\end{Lemma}
\begin{proof}
Recall the Lemma \ref{l8}, we define velocity $u(x,t)$ by setting $u(x,t)=m(x,t)/\rho(x,t)$ when $\rho(x,t)\neq 0$ and $u(x,t)=0$ when $\rho(x,t)=0$, we have
$$m(x,t)=\rho(x,t)u(x,t).$$
Moreover, Fatou's lemma yields
\begin{equation*}
\begin{aligned}
\int_{0}^{T}\int \rho u^{3}dxdt&\leq\int_{0}^{T}\int \lim_{\delta,r_{0}\rightarrow0}\inf \rho_{\delta,r_{0}}u_{\delta,r_{0}}^{3}dxdt\\
&\leq \lim_{\delta,r_{0}\rightarrow0}\inf\int_{0}^{T}\int \rho_{\delta,r_{0}}u_{\delta,r_{0}}^{3}dxdt,
\end{aligned}
\end{equation*}
hence, $\rho^{\frac{1}{3}}u\in L^{3}(0,T;L^{3})$.

Since $m_{\delta,r_{0}}\rightarrow m$ a.e. and $\rho_{\delta,r_{0}}\rightarrow\rho$ a.e., it's easy to show that
$$\sqrt{\rho_{\delta,r_{0}}}u_{\delta,r_{0}}\rightarrow m_{\delta,r_{0}}/\sqrt{\rho_{\delta,r_{0}}},\ \ a.e. \ in\ \{\rho(x,t)\neq 0\},$$
and for almost every $(x,t)$ in $\{\rho(x,t)=0\}$, we have
$$\sqrt{\rho_{\delta,r_{0}}}u_{\delta,r_{0}}\textbf{l}_{|u_{\delta,r_{0}}|\leq M}\leq M\sqrt{\rho_{\delta,r_{0}}}\rightarrow0,$$
as a matter of fact, $\sqrt{\rho_{\delta,r_{0}}}u_{\delta,r_{0}}\textbf{l}_{|u_{\delta,r_{0}}|\leq M}$ converges to $\sqrt{\rho}u\textbf{l}_{|u|\leq M}$ almost everywhere for $(x,t)$. Meanwhile, $\sqrt{\rho_{\delta,r_{0}}}u_{\delta,r_{0}}\textbf{l}_{|u_{\delta,r_{0}}|\leq M}$ is bounded in $L^{\infty}(0,T;L^{6})$, using the Egoroffs theorem gives
\begin{equation}\label{e8}
\sqrt{\rho_{\delta,r_{0}}}u_{\delta,r_{0}}\textbf{l}_{|u_{\delta,r_{0}}|\leq M}\rightarrow \sqrt{\rho}u\textbf{l}_{|u|\leq M}\ \ strongly\ \ in\ L^{2}(0,T;L^{2}).
\end{equation}

Since
\begin{equation}\label{e10}
\begin{aligned}
\int_{0}^{T}&\int|\sqrt{\rho_{\delta,r_{0}}}u_{\delta,r_{0}}-\sqrt{\rho}u|^{2}dxdt\\
&\leq \int_{0}^{T}\int|\sqrt{\rho_{\delta,r_{0}}}u_{\delta,r_{0}}\textbf{l}_{|u_{\delta,r_{0}}|\leq M}-\sqrt{\rho}u\textbf{l}_{|u|\leq M}|^{2}dxdt\\
&+2\int_{0}^{T}\int|\sqrt{\rho_{\delta,r_{0}}}u_{\delta,r_{0}}\textbf{l}_{|u_{\delta,r_{0}}|\geq M}|^{2}dxdt+2\int_{0}^{T}\int|\sqrt{\rho}u\textbf{l}_{|u|\geq M}|^{2}dxdt\\
&\leq\int_{0}^{T}\int|\sqrt{\rho_{\delta,r_{0}}}u_{\delta,r_{0}}\textbf{l}_{|u_{\delta,r_{0}}|\leq M}-\sqrt{\rho}u\textbf{l}_{|u|\leq M}|^{2}dxdt\\
&+\frac{2}{M}\int_{0}^{T}\int\rho_{\delta,r_{0}}u_{\delta,r_{0}}^{3}dxdt+\frac{2}{M}
\int_{0}^{T}\int\rho u^{3}dxdt\rightarrow0,
\end{aligned}
\end{equation}
as $r_{0}=\delta\rightarrow0$ and $M\rightarrow+\infty$.
Thus we proved that
$$\sqrt{\rho_{\delta,r_{0}}}u_{\delta,r_{0}}\rightarrow\sqrt{\rho}u\ \ strongly\ in\ L^{2}(0,T;L^{2}).$$
\end{proof}
\subsection{ Step 5: Convergence of the terms $r_{0}u_{\delta,r_{0}}$, $\rho_{\delta,r_{0}}\mathbb{D}u_{\delta,r_{0}}$, and  $\rho_{\delta,r_{0}}\nabla\triangle^{3}\rho_{\delta,r_{0}}$.}
Let $r_{0}=\delta$, since $\sqrt{r_{0}}u_{\delta,r_{0}}\in L^{2}(0,T;L^{2})$, for any test function $\varphi\in C_{per}^{\infty}((0,T);\mathbb{T}^{3})$, we have
\begin{equation*}
\begin{aligned}
r_{0}\int_{0}^{T}\int u_{\delta,r_{0}}\varphi dxdt\leq \sqrt{r_{0}}\|\sqrt{r_{0}}u_{\delta,r_{0}}\|_{L^{2}(L^{2})}\|\varphi\|_{L^{2}(L^{2})}\rightarrow
0,\ \ as\ r_{0}\rightarrow0.
\end{aligned}
\end{equation*}

To deal with the diffusion term $\rho_{\delta,r_{0}}\mathbb{D}u_{\delta,r_{0}}$, recall (\ref{b32}), we have
\begin{equation}\label{e11}
\begin{aligned}
&\int_{0}^{T}\int {\rm div}(\rho_{\delta,r_{0}}\mathbb{D}u_{\delta,r_{0}})\varphi=\int_{0}^{T}\int \partial_{i}(\rho_{\delta,r_{0}}(\frac{\partial_{i}u_{\delta,r_{0}}^{j}+\partial_{j}
u_{\delta,r_{0}}^{i}}{2}))\varphi\\
&=\frac{1}{2}\int_{0}^{T}\int (\rho_{\delta,r_{0}}u_{\delta,r_{0}}^{j})\partial_{ii}\varphi+\frac{1}{2}\int_{0}^{T}\int\partial_{i}
\rho_{\delta,r_{0}}u_{\delta,r_{0}}^{j}\partial_{i}\varphi+\frac{1}{2}\int_{0}^{T}\int (\rho_{\delta,r_{0}}u_{\delta,r_{0}}^{i})\partial_{ij}\varphi\\
&\ \ \ \ +\frac{1}{2}\int_{0}^{T}\int\partial_{j}
\rho_{\delta,r_{0}}u_{\delta,r_{0}}^{i}\partial_{i}\varphi,\\
&=\frac{1}{2}\int_{0}^{T}\int (\sqrt{\rho_{\delta,r_{0}}}\sqrt{\rho_{\delta,r_{0}}}u_{\delta,r_{0}}^{j})\partial_{ii}\varphi+\int_{0}^{T}\int\partial_{i}
\sqrt{\rho_{\delta,r_{0}}}\sqrt{\rho_{\delta,r_{0}}}u_{\delta,r_{0}}^{j}\partial_{i}\varphi\\
&\ \ \ \ +\frac{1}{2}\int_{0}^{T}\int (\sqrt{\rho_{\delta,r_{0}}}\sqrt{\rho_{\delta,r_{0}}}u_{\delta,r_{0}}^{i})\partial_{ij}\varphi
+\int_{0}^{T}\int\partial_{j}
\sqrt{\rho_{\delta,r_{0}}}\sqrt{\rho_{\delta,r_{0}}}u_{\delta,r_{0}}^{i}\partial_{i}\varphi,\\
\end{aligned}
\end{equation}
by using Lemma \ref{l6} and Lemma \ref{l9}, we can show
$$\int_{0}^{T}\int {\rm div}(\rho_{\delta,r_{0}}\mathbb{D}u_{\delta,r_{0}})\varphi\rightarrow
\int_{0}^{T}\int {\rm div}(\rho\mathbb{D}u)\varphi dxdt.$$
Finally, we show the convergence of the high order term $\rho_{\delta,r_{0}}\nabla\triangle^{3}\rho_{\delta,r_{0}}$:

Since
$$\delta^{\frac{5}{14}}\|\rho_{\delta,r_{0}}\|_{L^{\frac{14}{5}}(H^{3})}\leq \|\rho_{\delta,r_{0}}\|_{L^{\infty}(L^{3})}^{\frac{2}{7}}\|\sqrt{\delta}\rho_{\delta,r_{0}}\|
_{L^{2}(H^{4})}^{\frac{5}{7}}\leq C<+\infty,$$
which implies $\delta^{\frac{5}{14}}\rho_{\delta,r_{0}}\in L^{\frac{14}{5}}(0,T;H^{3})$.

For any test function $\varphi\in C_{per}^{\infty}([0,T];\mathbb{T}^{3})$, we have
$$\delta\int_{0}^{T}\int \rho_{\delta,r_{0}}\nabla\triangle^{3}\rho_{\delta,r_{0}} \varphi dxdt=-\delta\int_{0}^{T}\int\triangle{\rm div}(\rho_{\delta,r_{0}} \varphi)\triangle^{2}\rho_{\delta,r_{0}} dxdt,$$
we focus on the most difficult term
\begin{equation}\label{e12}
\begin{aligned}
&|\delta\int_{0}^{T}\int\triangle(\nabla\rho_{\delta,r_{0}})\triangle^{2}\rho_{\delta,r_{0}}\varphi
dxdt|\\
&\leq C\delta^{\frac{1}{7}}\|\sqrt{\delta}\triangle^{2}\rho_{\delta,r_{0}}\|_{L^{2}(L^{2})}\|\delta
^{\frac{5}{14}}\nabla^{3}\rho_{\delta,r_{0}}\|_{L^{\frac{14}{5}}(L^{2})}\|\varphi\|_{L^{7}(L^{\infty})}
\rightarrow0,
\end{aligned}
\end{equation}
as $\delta\rightarrow0$.

Similarly, we can deal with the other terms from
$$\delta\int_{0}^{T}\int\triangle{\rm div}(\rho_{\delta,r_{0}} \varphi)\triangle^{2}\rho_{\delta,r_{0}} dxdt.$$
Thus, we have
$$\delta\int_{0}^{T}\int \rho_{\delta,r_{0}}\nabla\triangle^{3}\rho_{\delta,r_{0}} \varphi dxdt\rightarrow0,$$
as $\delta\rightarrow0$.

 With all above compactness results, we can pass to the limits in (\ref{d7}) as $\delta\rightarrow0$, we have
 \begin{equation}\label{e13}
 \begin{aligned}
 &\rho_{t}+{\rm div}(\rho u)=0,\ holds\ in\ the \ sense\ of\ distribution\ on\ (0,T)\times\mathbb{T}^{3},\\
&(\rho u)_{t}+{\rm div}(\rho u\otimes u)+\nabla P(\rho)-{\rm div}(\rho \mathbb{D}u)+r_{1}\rho|u|u=\rho\nabla\Phi,\\
& \ \ \ \ \ \ \ \ \ \ \ \ \ \  \ \ \ \ \ \ \ \ \ \ \ \ \ \ \ \ \ \ \ holds\ in\ the \ sense\ of\ distribution\ on\ (0,T)\times\mathbb{T}^{3},\\
&\triangle\Phi=-4\pi G(\rho-1),\ \ holds\ a.e.\ on\ (0,T)\times\mathbb{T}^{3}.
 \end{aligned}
 \end{equation}

 Furthermore, due to the lower semi-continuity of the convex functions ,we can obtain the following energy inequality and B-D entropy by passing to the limits as $r_{0}=\delta\rightarrow0$:
 \begin{equation}\label{e14}
\begin{aligned}
\int\frac{1}{2}\rho &u^{2}+\frac{1}{a\gamma(\gamma-1)}\rho^{\gamma}dx+\int_{0}^{T}\int\rho |\mathbb{D}u|^{2}+r_{1}\int_{0}^{T}\int\rho|u|^{3}
\leq C(T),
\end{aligned}
\end{equation}
and
\begin{equation}\label{e15}
\begin{aligned}
&\frac{1}{2}\int \rho(u+\frac{\nabla\rho}{\rho})^{2} dx+\frac{1}{2}\int_{0}^{T}\int \rho|\nabla u|^{2} +\frac{2}{a\gamma^{2}}\int_{0}^{T}\int|\nabla\rho^{\frac{\gamma}{2}}|^{2}\leq C(T).
\end{aligned}
\end{equation}
Thus we have completed the proof of the Theorem \ref{theorem1}.

{\bf Acknowledgement.} This work was completed while the author was visiting the
Department of Mathematics at The University of Texas at Austin. The author are grateful for
many discussions with Prof. Alexis F.Vasseur and Cheng Yu.



\vspace{2mm}



\begin{thebibliography}{00}
\bibitem{a1963}
 J.P.Aubin, Un th$\acute{e}$or$\grave{e}$me de compacit$\acute{e}$. (French) C. R. Acad. Sci. Paris 256 1963 5042-5044.
\bibitem{BDL2003}
D.Bresch, B.Dejardins, Chi-Kun lin, On some compressible fluid models: Korteweg, lubrication, and shallow water systems. Comm. Partial Differential Equations 28 (2003), no. 3-4, 843-868.
\bibitem{bdd2005}
D.Bresch, B.Desjardins, B.Ducomet, Quasi-neutral limit for a viscous capillary model of plasma. Ann. Inst. H. Poincar¨¦ Anal. Non Lin¨¦aire 22 (2005), no. 1, 1-9.
\bibitem{ct2015}
H.Cai, Z.Tan, Weak time-periodic solutions to the compressible Navier-Stokes-Poisson equations. Commun. Math. Sci. 13 (2015), no. 6, 1515-1540.
\bibitem{D2003}
D.Donatelli, Local and global existence for the coupled Navier-Stokes-Poisson problem. Quart. Appl. Math. 61 (2003), no. 2, 345-361.
\bibitem{DF2004}
B.Documet, E.Feireisl, On the dynamics of gaseous stars. Arch. Ration. Mech. Anal. 174 (2004), no. 2, 221-266.
\bibitem{DFPS2004}
B.Documet, E.Feireisl, H.Petzeltov$\acute{a}$, I.Stra$\breve{s}$kraba,  Global in time weak solutions for compressible barotropic self-gravitating fluids. Discrete Contin. Dyn. Syst. 11 (2004), no. 1, 113-130.
\bibitem{dnv2010}
 B.Ducomet,$\check{S}$.Ne$\check{c}$asov$\acute{a}$, A.Vasseur, On global motions of a compressible barotropic and selfgravitating gas with density-dependent viscosities. Z. Angew. Math. Phys. 61 (2010), no. 3, 479-491.
 \bibitem{dnv2011}
 B.Ducomet,$\check{S}$.Ne$\check{c}$asov$\acute{a}$, A.Vasseur, On spherically symmetric motions of a viscous compressible barotropic and selfgravitating gas. J. Math. Fluid Mech. 13 (2011), no. 2, 191-211.
\bibitem{feireisl2004}
E.Feireisl, Dynamics of Viscous Compressible Fluids, Oxford University Press, Oxford, 2004
\bibitem{fn2009}
E.Feireisl, A. Novotn$\acute{y}$, Singular Limits in Thermodynamics of Viscous Fluids, Adv. Math. Fluid Mech., Birkh$\ddot{a}$user Verlag, Basel, 2009.
\bibitem{fnp2001}
E.Feireisl, A. Novotn$\acute{y}$, H.Petzeltov$\acute{a}$, On the existence of globally defined weak solutions to the Navier-Stokes equations. J.Math.Fluid Mech. 3 (2001), 358-392.
\bibitem{lx2015}
J.Li, Z.P.Xin, Global Existence of Weak Solutions to the Barotropic Compressible Navier-Stokes Flows with Degenerate Viscosities.  arXiv:1504.06826.
\bibitem{lions1998}
P.-L. Lions, Mathematical topics in fluid mechanics. Vol. 2. Compressible models. Oxford Lecture Series in Mathematics and its Applications, 10. Oxford Science Publications. The Clarendon Press, Oxford University Press, New York, 1998.
\bibitem{jungel2010}
A.J$\ddot{u}$ngel, Global weak solutions to compressible Navier-Stokes equations for quantum fluids. SIAM J. Math.Anal.42 (2010), no.3, 1025-1045.
\bibitem{jty2009}
 F.Jiang, Z.Tan, QL.Yan, Asymptotic compactness of global trajectories generated by the Navier-Stokes-Poisson equations of a compressible fluid. NoDEA Nonlinear Differential Equations Appl. 16 (2009), no. 3, 355-380.
 \bibitem{mv2007}
  A.Mellet, A.Vasseur, On the barotropic compressible Navier-Stokes equations. Comm. Partial Differential Equations 32 (2007), no. 1-3, 431-452.
  \bibitem{n1959}
  L.Nirenberg, On elliptic partial differential equations. Ann. Scuola Norm. Sup. Pisa (3) 13 1959 115-162.
  \bibitem{nsbook}
  A.Novotn$\acute{y}$, I.Stra$\check{s}$kraba, Introduction to the mathematical theory of compressible flow. Oxford Lecture Series in Mathematics and its Applications, 27. Oxford University Press, Oxford, 2004. xx+506 pp. ISBN: 0-19-853084-6.
  \bibitem{s1986}
   J.Simon, Compact sets in the space $L^{p}(0,T;B)$. Ann. Mat. Pura Appl. (4) 146 (1987), 65-96.
  \bibitem{vy2014}
  A.Vasseur, C.Yu, Existence of Global Weak Solutions for 3D Degenerate Compressible Navier-Stokes Equations.  arXiv:1501.06803.
\bibitem{vy2015}
A.Vasseur, C.Yu, Global weak solutions to compressible quantum Navier-Stokes equations with damping.  arXiv:1503.06894.
\bibitem{yj2014}
YW.Yang, QC.Ju, Existence of global weak solutions for Navier-Stokes-Poisson equations with quantum effect and convergence to incompressible Navier-Stokes equations. Math. Methods Appl. Sci. DOI: 10.1002/mma.3304
\bibitem{zt2007}
YH.Zhang, Z.Tan, On the existence of solutions to the Navier-Stokes-Poisson equations of a two-dimensional compressible flow. Math. Methods Appl. Sci. 30 (2007), no. 3, 305-329.
\end{thebibliography}
\end{document}